\documentclass[a4paper]{article}
\usepackage{amsmath}
\usepackage{mathrsfs}
\usepackage{amsfonts}
\usepackage{graphicx}
\usepackage{amsthm}
\allowdisplaybreaks
\usepackage{amssymb}
\usepackage{stmaryrd}
\usepackage[all]{xypic}
\usepackage[all]{xy}
\usepackage[toc,page,title,titletoc,header]{appendix}
\usepackage{xcolor}
\usepackage{hyperref}
\usepackage[square,numbers]{natbib}
\usepackage{fancyhdr}
\usepackage[a4paper, margin=3cm]{geometry}
\usepackage[capitalize]{cleveref}

\newcommand{\R}{{\mathbb R}}

\newcommand{\Z}{{\mathbb Z}}
\newcommand{\N}{{\mathbb N}}

\newcommand{\Spec}{\mathrm{Spec}}

\newcommand{\Supp}{\mathrm{Supp}}

\newcommand{\OO}{\mathcal{O}}

\newcommand{\id}{\operatorname{id}\nolimits}

\newcommand{\Hom}{\operatorname{Hom}\nolimits}

\newcommand{\Ker}{\operatorname{Ker}\nolimits}

\renewcommand{\dim}{\operatorname{dim}\nolimits}

\newcommand{\rk}{\mathrm{rk}}

\newcommand{\pr}{\mathrm{pr}}

\newcommand{\trop}{\mathrm{trop}}
\newcommand{\an}{\mathrm{an}}

\newtheorem{theorem}{Theorem}[section]

\newtheorem{corollary}[theorem]{Corollary}

\newtheorem{definition}[theorem]{Definition}
\newtheorem{example}[theorem]{Example}
\newtheorem{example*}{Example}

\newtheorem{lemma}[theorem]{Lemma}

\newtheorem{proposition}[theorem]{Proposition}
\newtheorem{remark}[theorem]{Remark}
\newtheorem{lemma*}{Lemma}
\newtheorem{question*}{Question}

\newtheorem{proposition&definition}[theorem]{Proposition\&Definition}
\newtheorem{lemma&definition}[theorem]{Lemma\&Definition}
\newtheorem{theorem&definition}[theorem]{Theorem\&Definition}

\begin{document}
	
	\title
	{The push-forwards and pull-backs of $\delta$-forms and applications to non-archimedean Arakelov geometry}
	
	\author{Yulin Cai}
	
	\newcommand{\address}{{
			\bigskip
			\footnotesize
			Y.~Cai,
			\textsc{Westlake University, 
				SanDun Road 600, Xihu District
				310024, Hangzhou, China}\par\nopagebreak
			
			\textit{E-mail address}: \texttt{caiyulin@westlake.edu.cn}
	}}

	\maketitle
	
	\begin{abstract}
	We study two kinds of push-forwards of $\delta$-forms and define the pull-backs of $\delta$-forms. As a generalization of Gubler-K\"unnemann, we prove the projection formula and the tropical Poincar\'e-Lelong formula. As an application, we follow the idea in \cite{gubler2017a} and generalize the notion of $\delta$-forms on algebraic varieties, this allows us to define the first Chern forms for any piecewise smooth metrics.
	\end{abstract}
	
	{\footnotesize
		
		\tableofcontents

	}

\section{Introduction}

In \cite{chambert2012formes}, Chambert-Loir and Ducros introduce a potential theory for an analytic space $X$ (not necessarily algebraic or compact) over a complete non-archimedean field $K$ in sense of Berkovich. Such theory deals with the non-archimedean part of Arakelov theory in a similar way to the archimedean part. With this theory, they show that the Chambert-Loir measure \cite{chambert2006measures} is a limit of the wedge product of first Chern forms of smooth metrized line bundles. 

In \cite{gubler2017a}, Gubler and K\"unnemann give the theory of $\delta$-forms and $\delta$-currents on algebraic varieties. With their theory, they define the first Chern form $c_1(\overline{L})$ for a metrized line bundle $\overline{L}$, then the measure $c_1(\overline{L})^n$ comes naturally. Moreover, they define the first Chern form for a so-called $\delta$-metric, which is a special piecewise smooth metric. They show that the canonical metrics of line bundles on abelian varieties are $\delta$-metrics.

Inspired by Gubler and K\"unnemann's paper~\cite{gubler2017a}, Mihatsch introduces the $\delta$-forms on tropical geometry and construct intersection product in \cite{mihatsch2021on}. In his another recent paper \cite{mihatsch2021delta}, he considers $\delta$-forms on tropical spaces. Since the skeletons of good Berkovich spaces satisfy a tropical balance condition, he is able to define $\delta$-forms on good, strict Berkovich spaces.
	
The main object in this paper is to develop the theory of $\delta$-forms on tropical geometry, including the push-forward and pull-back and some important formulas: projection formula, tropical Poincar\'e-Lelong formula. As an application, following the idea in \cite{gubler2017a}, we can have more $\delta$-forms on algebraic varieties comparing to \cite{gubler2017a}, and we can define the first Chern form for any piecewise smooth metrics. Using the tropical Poincar\'e-Lelong formula, we prove the Poincar\'e-Lelong formula for $\delta$-currents.

In \cref{superforms on N}, \cref{polyhedral supercurrents on N},  we recall superforms, supercurrents and polyhedral supercurrents on affine spaces for readers' convenience.

We will recall the main results of Mihatsch \cite{mihatsch2021on} in \cref{delta forms on N}, including \cref{main theorem of delta forms} which asserts that a polyhedral supercurrent is a $\delta$-form if and only if it is balanced, and \cref{intersection} which shows that there exists a unique $\wedge$-product of $\delta$-forms. One interesting fact is that the $\wedge$-product extends both that of superforms and intersection product of tropical cycles. For a finitely generated free abelian group $N$ and an open subset $\widetilde{\Omega}\subset N_\R$, the set of $\delta$-forms of $(p,q,l)$-type (resp. piecewise smooth forms of $(p,q)$-type) on $\widetilde{\Omega}$  is denoted by $B^{p,q,l}(\widetilde{\Omega})$ (resp. $PS^{p,q}(\widetilde{\Omega})$). The notion of corner locus $\mathrm{div}(\phi)\cdot \alpha$ is given for $\alpha\in B^{p,q,l}(\widetilde{\Omega})$ and $\phi\in PS^{0,0}(\widetilde{\Omega})$ in \cref{corner locus}. Then we show the following tropical Poincar\'e-Lelong formula which generalizes \cite[Corollary~3.19]{gubler2017a}.

\begin{theorem}[\cref{tropical poincare-lelong formula}, Tropical Poincar\'e-Lelong formula]
	For any $\alpha\in B^{p,q,l}(\widetilde{\Omega})$ and $\phi\in PS^{0,0}(\widetilde{\Omega})$, we have 
	\[\mathrm{div}(\phi)\cdot \alpha =  d'd''\phi\wedge\alpha-d_P'd_P''\phi\wedge\alpha= \partial'(d_P''\phi\wedge\alpha)+d_P''\phi\wedge\partial'\alpha\]
	in $P(\widetilde{\Omega})$. In particular, $\mathrm{div}(\phi)\cdot \alpha\in B^{p,q,l+1}(\widetilde{\Omega})$.
\end{theorem}
Here $d_P', d_P''$ (resp. $\partial', \partial''$) are polyhedral derivatives (resp. boundary derivatives), and $d'=d_P'+\partial', d''=d''_P+\partial''$. 

In \cref{pushforwardandpullback}, we consider two kinds of push-forwards of $\delta$-forms via integral $\R$-affine maps, one is from supercurrent point of view, the other is from tropical cycle point of view. We define the pull-backs of $\delta$-forms, this is crucial for us to define $\delta$-forms on algebraic varieties. Moreover, we prove the following result.
\begin{theorem}[\cref{projection formula for delta forms}]
	The presheaf $B: \mathcal{AT}op_{\Z,\R}\rightarrow \mathcal{R}ing, \ \ (N,\widetilde{\Omega})\mapsto B(\widetilde{\Omega})$ is a $C^\infty$-module. Moreover, for a morphism $F: (N',\widetilde{\Omega}')\rightarrow (N,\widetilde{\Omega})$, the following statements hold.  
	\begin{enumerate}
		\item[(1)] The $\wedge$-product on $B(\widetilde{\Omega})$ satisfies Leibniz rules with respect to $d',d'', d'_P, d''_P, \partial'$ and $\partial''$.
		\item [(2)] $F^*$ commutes with $d', d'', d_P, d''_P, \partial',\partial''$, and it is graded, i.e. $F^*(B^{p,q,l}(\widetilde{\Omega}))\subset B^{p,q,l}(\widetilde{\Omega}')$,.
		\item[(3)] (Projection formula for $F_*$) for any $\alpha\in B_c(\widetilde{\Omega}'), \beta\in B(\widetilde{\Omega})$, we have
\[F_*(\alpha\wedge F^*\beta) = F_*\alpha\wedge\beta \in B(\widetilde{\Omega}).\]
\item[(4)] (Projection formula for $\widehat{F}_*$) For any $\alpha\in B(N_\R'), \beta\in B(N_\R)$, we have
\[\widehat{F}_*(\alpha\wedge F^*\beta) = \widehat{F}_*\alpha\wedge\beta \in B(N_\R).\]
	\end{enumerate}
\end{theorem}

In Section~\ref{charts on an algebraical variety}, we review the concept of tropicalization and canonical charts on algebraic varieties. Then we define the $\delta$-forms following the idea in \cite{gubler2017a}. Our definition gives more $\delta$-forms such that we have differentials without taking quotient as in \cite[4.6]{gubler2017a}. Most of concepts and results, e.g. integration, stokes' formula, can be given in a similar way. 


In \cref{The Poincare-Lelong formula and the first Chern delta-forms}, we reprove the Poincar\'e-Lelong formula by the tropical Poincar\'e-Lelong formula.
\begin{theorem}[\cref{poincare-lelong equation},The Poincar\'e-Lelong formula]
	Let $X$ be an integral algebraic variety over $K$. For any rational function $f\in K^*(X)$, the Poincar\'e-Lelong equation\[\delta_{[\mathrm{div}(f)]} = d'd''[\log|f|]\]
	holds in $E^{1,1}(X^\an)$.
	\label{poincare-lelong equation in introduction}
\end{theorem}
Notice that our method is different from \cite{chambert2012formes} and \cite{gubler2017a}, we use \cite[Proposition~4.6.6]{chambert2012formes} and the tropical Poincar\'e-Lelong formula. This theorem allows us to define the first Chern $\delta$-forms $c_1(\overline{L})$ of a metrized line bundle $\overline{L}$ with piecewise smooth metric $||\cdot||$, this generalizes the same notion for $\delta$-metrics in \cite[Section~9]{gubler2017a}.

\subsection*{Notation and terminology}

For terminology in convex geometry, we refer to \cite[Appendix]{gubler2017a} and \cite{mikhalkin2018tropical}. We write $N, N',\cdots$ for finitely generated free abelian groups and $N^*, (N')^*\cdots$ for their dual groups. The rank of $N$ is denoted by $\rk(N)$. Without confusion, we use the notation $N_\R:=N_\R\otimes_\Z\R$ both for the $\R$-linear space and the corresponding affine space. We write $\sigma,\tau,\cdots$ for integral $\R$-affine polyhedra on $N_\R$. For a given integral $\R$-affine polyhedron $\sigma$, we denote $\mathbb{A}_{\sigma}$ the affine subspace generated by the polyhedron $\sigma$, and $N_{\sigma,\R}$ the underlying vector space of $\mathbb{A}_{\sigma}$. We set $N_\sigma=N\cap N_{\sigma,\R}$. We also write $\mathrm{relint}(\sigma)$ for the set of internal of $\sigma$ in $\mathbb{A}_\sigma$ and $\partial\sigma=\sigma\setminus\mathrm{relint}(\sigma)$.

Given an integral $\R$-affine polyhedral complex $\mathcal{C}$ in $N_\R$ and $\sigma, \tau\in \mathcal{C}$, we write $\tau\leq\sigma$ if $\tau$ is a face of $\sigma$, write $\tau<\sigma$ if moreover $\tau\not=\sigma$. We call $\tau$ a facet of $\sigma$ if $\tau<\sigma$ and $\dim(\sigma)-\dim(\tau)=1$. We denote $\mathcal{C}^l$ (resp. $\mathcal{C}_n$) the polyhedra in $\mathcal{C}$ of codimension $l$ (resp. dimension $n$) with respect to $N_\R$. A smooth weight on $\mathcal{C}$ is a family of functions $\{m_\sigma\}_\sigma$ with $m_\sigma$ a smooth function on $\mathrm{relint}(\sigma)\subset \mathbb{A}_\sigma$, where $\sigma$ runs through top-dimension polyhedra in $\mathcal{C}$. The refinement relation is an equivalence relation on the set of integral $\R$-affine complexes with constant weights.  We will write $C=\sum\limits_{\sigma}m_\sigma[\sigma]$ for an equivalence class, and $|C|=\bigcup\limits_{\sigma}\Supp(\sigma)$. Furthermore, we denote the set of equivalence classes of pure dimension $n$ by $F_n(N_\R)$, and $F^l(N_\R):= F_{\rk(N)-l}(N_\R)$. 

For an affine map $F: N_\R'\rightarrow N_\R$, we denote the corresponding linear map by $dF: N_\R'\rightarrow N_\R$. For convenience, we set the $\mathcal{AT}op_{\Z,\R}$ the category of open subsets of affine spaces, i.e. objects are $(N,\widetilde{\Omega})$ with $\widetilde{\Omega}\subset N_\R$ and morphisms are  integral $\R$-affine maps $F:(N',\widetilde{\Omega}')\rightarrow (N, \widetilde{\Omega})$ such that $F(\widetilde{\Omega}')\subset\widetilde{\Omega}$.  

Throughout this paper, $K$ is a complete non-archimedean field with non-trivial valuation. For an algebraic variety $X$ over $K$, we mean a separated scheme of finite type over $K$. We denote $X^\an$ the analytification of $X$ in sense of Berkovich \cite{berkovich1990spectral}.


\section{Superforms and supercurrents on $N_\R$}

In this section, we will recall superforms and supercurrents on affine spaces. The notion is firstly introduced by Lagerberg~\cite{lagerberg2012super}. The references \cite[Section~1]{chambert2012formes} and \cite[Section~2]{gubler2016forms} also give a complete discussion of superforms and supercurrents. 

Throughout this section, $N$ and $N'$ denote free abelian groups of finite rank $r$ and $r'$. Since most of results in this section are not hard to prove, we will leave them to readers if we don't prove here.

\label{superforms on N}

\subsection{Superforms}

\begin{definition}
	For an open subset $U\subset N_\R$, a {\bf $(p,q)$-superform} on $U$ is an element in 
	\[A^{p,q}(U): = C^\infty(U)\otimes\wedge^{p}(N_\R)^*\otimes\wedge^q(N_\R)^*.\]
	Formally, a $(p,q)$-superform can be written as 
	\[\alpha = \sum\limits_{|I|=p, |J|=q}\alpha_{IJ}d'x_I\wedge d''x_J,\]
	where $I$ (resp. $J$) consists of $0<i_1<\cdots < i_p\leq r$ (resp. $0<j_1<\cdots< j_q\leq r$), $\alpha_{IJ}\in C^\infty(U)$ and
	\[d'x_I\wedge d''x_J: = d'x_{i_1}\wedge\cdots \wedge d'x_{i_p}\otimes d''x_{j_1}\wedge\cdots \wedge d''x_{j_p}.\]
	We denote \[A(U):=\bigoplus\limits_{p,q=1}^nA^{p,q}(U),\] which can be viewed as a $C^\infty(U)$-module.
	The {\bf wedge product} is defined in the usual way in $A(U)$.

	There is a {\bf differential operator} $d': A^{p,q}(U)\rightarrow A^{p+1,q}(U)$ defined by
	\[d'\alpha = \sum\limits_{|I|=p, |J|=q}\sum\limits_{i=1}^r\frac{\partial\alpha_{IJ}}{\partial x_i}d'x_i\wedge d'x_I\wedge d''x_J.\]
	Similarly, we define  $d'': A^{p,q}(U)\rightarrow A^{p,q+1}(U)$ by
	\begin{align*}
		d''\alpha &=\sum\limits_{|I|=p, |J|=q}\sum\limits_{j=1}^r\frac{\partial\alpha_{IJ}}{\partial x_j}d''x_j\wedge d'x_I\wedge  d''x_J\\
		&=\sum\limits_{|I|=p, |J|=q}\sum\limits_{j=1}^r(-1)^p\frac{\partial\alpha_{IJ}}{\partial x_j} d'x_I\wedge d''x_j\wedge d''x_J.
	\end{align*}
	Moreover, we set $d=d'+d''$.
	
	We define the operation ${J}: A^{p,q}(U)\rightarrow A^{q,p}(U)$ by 
	\[{J}\alpha = (-1)^{pq}\sum\limits_{|I|=p, |J|=q}\alpha_{IJ}\wedge d'x_J\wedge  d''x_I.\]
	A superform $\alpha\in A^{p,p}(U)$ is called {\bf symmetric} if $\alpha = (-1)^pJ\alpha$.
	
	Obviously, $A^{p,q}(U)$ defines a sheaf ${A}^{p,q}$ on $N_\R$ with respect to the general topology. We denote  the space of $(p,q)$-superforms on $U$ with compact support by $A_c^{p,q}(U)$.
	\label{definition of superform}
\end{definition}
\begin{remark}
	\begin{enumerate}
		\item[(1)] We have obvious properties: $d'd'=d''d''=0$. For $\alpha\in A^{p,q}(U), \beta\in  A^{p',q'}(U)$, we have
		\[d'(\alpha\wedge\beta) = d'\alpha\wedge\beta +(-1)^{p+q}\alpha\wedge d'\beta.\]
		\item [(2)] For an affine map $F: N_\R'\rightarrow N_\R$ and any open subset $U\subset N_\R$, we have a pull-back
		\begin{align*}
			F^*: A^{p,q}(U)&\rightarrow A^{p,q}(F^{-1}(U)),\\
			d'x_I\wedge d"x_J&\mapsto \sum\limits_{|\widetilde{I}|=p, |\widetilde{J}|=q}F_{I\widetilde{I}}F_{J\widetilde{J}}d'x_{\widetilde{I}}'\wedge d"x_{\widetilde{J}}',
		\end{align*}
		which commutes with $d', d''$, where $F_{I\widetilde{I}}=\det(F(I,\widetilde{I}))\in \Z$ is the $(I,\widetilde{I})$-th minor of the linear map $dF: N'_\R\rightarrow N_\R$. 
		
		\item[(3)] In convention, for a smooth function $f\in C^\infty(U)$ and $v = \sum\limits_{i=1}^ra_ie_i\in N_\R$, where $e_1,\cdots, e_r$ is a basis of $N$ corresponding to the linear functions $x_1,\cdots, x_r$, then we denote 
		\[\frac{\partial f}{\partial v}:= \sum\limits_{i=1}^ra_i\frac{\partial f}{\partial x_i}.\]
	\end{enumerate}
\end{remark}

\subsection{Integration}

\begin{definition}
	Let $x_1,\cdots, x_r \in N^*$ be the coordinates for some basis of $N$. For any open subset $U\subset N_\R$ and $\alpha=\alpha_{LL}d'x_L\wedge d''x_L\in A_c^{r,r}(U)$ with $L=\{1,\cdots, r\}$, we  set 
	\[\int_U\alpha: = (-1)^{r(r-1)/2}\int_U\alpha_{LL}dx_1\wedge\cdots\wedge dx_r=(-1)^{r(r-1)/2}\int_U\alpha_{LL}dx_1\cdots dx_r\]
	where the right-hand side is the integration on $\R^r$.
	
	For a polyhedron $\sigma\subset U$ of dimension $r$, we set
	\[\int_\sigma\alpha: = (-1)^{r(r-1)/2}\int_\sigma\alpha_{LL}dx_1\wedge\cdots\wedge dx_r=(-1)^{r(r-1)/2}\int_\sigma\alpha_{LL}dx_1\cdots dx_r.\]
	
	\label{definition of integration}
\end{definition}
\begin{remark}
	\begin{enumerate}
		\item[(1)] This definition is independent of the choice of orientation.
	\end{enumerate}
\end{remark}

The following proposition shows that the definition of integration  is independent of the choice of basis of $
N$.

\begin{proposition}
	Assume that $N, N'$ are both of rank $r$. Let $F: N'_\R\rightarrow N_\R$ be an $\R$-linear isomorphism. Then for any open subset $U\subset N_\R$ and $\alpha\in A_c^{r,r}(U)$
	\[\int_{F^{-1}(U)}F^*\alpha = |\det(dF)|\int_U\alpha.\] 
	where $dF:N_\R'\rightarrow N_\R$ the linear map on the tangent spaces. 
	\label{integral via bijection}
\end{proposition}

\

Integration can be defined on a polyhedron complex with weights of pure top dimension.

\begin{definition}
	Let $C=\sum\limits_{\sigma\in\mathcal{C}^0}m_\sigma[\sigma]\in F^{0}(N_\R)$ be a polyhedron complex of pure dimension $r$ with constant weights. For an $(r,r)$-superform $\alpha$ is defined on a neighborhood of $|\mathcal{C}|$ with compact support,  we set
	\[\int_{C}\alpha: = \sum\limits_{\sigma\in \mathcal{C}^0}m_\sigma\int_{\sigma}\alpha, \]
	here we use integration from \cref{definition of integration} on the right. Obviously, this is independent of the choices of representatives of $C$.
	\label{definition of integration of superforms on polyhedron complex with weights}
\end{definition}

For further definition of more general integration, we use contraction.

\begin{definition}
	For any open subset $U\subset N_\R$, $\alpha\in A^{p,q}(U)$, and $((v_i)_{i\in I}, (w_j)_{j\in J})$ with $I\subset \{1,\cdots, p\}, J\subset \{1,\cdots, q\}$ and $v_i, w_j\in N_\R$, the {\bf contraction} $\langle\alpha; ((v_i)_{i\in I}, (w_j)_{j\in J})\rangle \in A^{p-|I|, q-|J|}(U)$ is given by 
	in inserting $v_i$, $i\in I$ and $w_j$, $j\in J$ at corresponding positions. 
	\label{contraction}
\end{definition}
		$\alpha: N_\R^{p+q}\rightarrow C^\infty(U)$.

\begin{lemma}
	Keep the notion in \cref{contraction}. 
	\begin{enumerate}
		\item [(1)](Leibniz rule) For any $\alpha\in A^{p,q}(U), \beta\in  A^{p',q'}(U)$ and $v\in N_\R$, we have
		\[\langle\alpha\wedge \beta; ((v)_{\{1\}}, \emptyset)\rangle = \langle\alpha; ((v)_{\{1\}}, \emptyset)\rangle\wedge\beta+ (-1)^{p+q}\alpha\wedge\langle\beta; ((v)_{\{1\}}, \emptyset)\rangle,\]
		\[\langle\alpha\wedge \beta; (\emptyset, (v)_{\{1\}})\rangle = \langle\alpha; (\emptyset, (v)_{\{1\}})\rangle\wedge\beta+ (-1)^{p+q}\alpha\wedge\langle\beta; (\emptyset,(v)_{\{1\}})\rangle.\]
		\item[(2)] For any $f\in C^\infty(U)$ and $v\in N_\R$, we have
		\[\langle d'f; ((v)_{\{1\}}, \emptyset)\rangle = \langle d''f; (\emptyset, (v)_{\{1\}})\rangle= \frac{\partial f}{\partial v}.\]
	\end{enumerate}
	\label{basic properties of contraction}
\end{lemma}
\begin{proof}
	(1) By linearity, we can assume that $\alpha=d'x_{I'}\wedge d''x_{J'}$, $\beta = d'x_{I''}\wedge d''x_{J''}$ and $v=e_i\in N_\R$ corresponding to the function $x_i$. Then the result is based on the following fact
	\[\langle d'x_I\wedge d''x_J;((e_i)_{\{1\}},\emptyset)\rangle=\begin{cases}
		0& \text{ if $i\not\in I$,}\\
		(-1)^{l-1}d'x_{I\setminus i}\wedge d''x_J& \text{ if $i = i_l\in I$.}
	\end{cases}\]
	We check the case where $i\in I'$ and $i\in I''$. In this case, in the first equality of (1), the left-hand side is $0$ and the right hand side is
	\begin{align*}
		&(-1)^{l'-1}d'x_{I'\setminus i;}\wedge d''x_{J'}\wedge dx_{I''}\wedge d''x_{J''} + (-1)^{p+q}d'x_{I'}\wedge d''x_{J'}\wedge (-1)^{l''-1}d'x_{I''\setminus i}\wedge d''x_{J''}\\=& ((-1)^{l'+qp'-1}d'x_{I'\setminus i;}\wedge dx_{I''} + (-1)^{p+qp'+l''-1}d'x_{I'}\wedge d'x_{I''\setminus i})\wedge d''x_{J'}\wedge d''x_{J''}\\
		=&((-1)^{l'+qp'-1}+(-1)^{p+qp'+l''-1+(p-l'+l''-1)})d'x_{I'\setminus i;}\wedge dx_{I''}\wedge d''x_{J'}\wedge d''x_{J''}\\
		=&0.
	\end{align*}
	
	(2) By definition and linearity, 
	we can assume $v=e_j$. Then the result is obvious.
\end{proof}

\subsection{Stokes' and Green's formula}


Recall, for an integral $\R$-affine polyhedron $\sigma$ and one of its facet $\tau$, we choose a representative $\omega_{\sigma,\tau}\in N_\sigma$ of the generator of the one-dimensional lattice $N_\sigma/N_\tau$ pointing the direction of $\sigma$, such vector $\omega_{\sigma,\tau}$ is called a normal vector for $\sigma/\tau$.

\begin{definition}
	Let $\sigma$ be an $r$-dimensional integral $\R$-affine polyhedron in $N_\R$. For any closed facet $\tau$ of $\sigma$, let $\omega_{\sigma,\tau}\in N_\sigma = N$ be a normal vector for $\sigma/\tau$. For any $(r-1,r)$-superform $\eta$ with compact support on an open neighborhood of $\sigma$, we have the contraction $\langle\eta;(\emptyset,(\omega_{\sigma,\tau})_{\{r\}})\rangle = (-1)^{r-1}\langle\eta;(\emptyset,(\omega_{\sigma,\tau})_{\{1\}})\rangle$, which is an $(r-1,r-1)$-form. Then we set
	\[\int_{\partial\sigma}\eta: = (-1)^r\sum\limits_{\tau\leq \sigma \text{ a facet}}\int_\tau\langle\eta;(\emptyset,(\omega_{\sigma,\tau})_{\{1\}})\rangle, \]
	where the right-hand side is defined in \cref{definition of integration} on $N_{\tau,\R}$. The integration is independent of the choice of $\omega_{\sigma,\tau}$. 
	
	Similarly, any $(r,r-1)$-superform $\eta$ with compact support on an open neighborhood of $\sigma$, we set
	\[\int_{\partial\sigma}\eta: = \sum\limits_{\tau\leq \sigma \text{ a facet}}\int_\tau\langle\eta;((\omega_{\sigma,\tau})_{\{1\}},\emptyset)\rangle. \]
	
	For an integral $\R$-affine polyhedral complex $C=\sum\limits_{\sigma\in\mathcal{C}^0}m_\sigma[\sigma]\in \in F^{0}(N_\R)$ of constant weight, we can define $\int_{\partial C}\eta$ by linearity. 
	This is independent of the choices of representatives of $C$.
	\label{definition of integration on boundary for superforms}
\end{definition} 
\begin{remark}
	\begin{enumerate}
		\item [(1)] Note that for any $(r-1,r)$-superform $\eta$ around $\sigma$ and a facet $\tau\leq \sigma$, the form  $\langle\eta;(\emptyset,(\omega_{\sigma,\tau})_{\{r\}})\rangle|_\tau$ is independent of the choice of $\omega_{\sigma,\tau}$. It is similar for $\langle\eta;((\omega_{\sigma,\tau})_{\{r\}},\emptyset)\rangle|_\tau$ with $\eta$ an $(r,r-1)$-superform around $\sigma$.
		
		\begin{proof}
			This is from the fact that for any $v_1,\cdots, v_r\in N_\tau$, we have 
			\[\langle\eta;(\emptyset,(v_j)_{j\in\{1,\cdots, r\}})\rangle|_\tau = 0.\]
		\end{proof}
		
	\end{enumerate}
\end{remark}

\begin{proposition}[Stokes' formula]
	Let  $C=\sum\limits_{\sigma\in\mathcal{C}^0}m_\sigma[\sigma]\in F^0(N_\R)$ be an integral $\R$-affine polyhedron complex with constant weights. For any $(r-1,r)$-superform $\eta'$ (resp. $(r,r-1)$-superform $\eta''$) with compact support on around $|C|$, we have
	\[\int_{C}d'\eta' = \int_{\partial C}\eta',\]
	\[\int_{C}d"\eta'' = \int_{\partial C}\eta''.\]
	\label{stokes' formula for polyhedra}
\end{proposition}
\begin{proof}
	Only consider the case where $C=[\sigma]$ is a polyhedron of dimension $r$. To apply the classical Stokes' formula for polyhedrons ($p$-chains), we fix coordinate functions $x_1,\cdots, x_r$ of $N_\R$, and assume $\sigma$ is an $r$-simplicial complex, this can be realized by triangulation (the integration on the new boundaries can be canceled). We have an isomorphism $\varphi:N_\R\simeq \R^r$, \[\eta' = \sum\limits_{i=1}^r\eta_{i}'d'x_1\wedge\cdots\wedge \widehat{d'x_i}\wedge\cdots\wedge d'x_r\wedge d''x_1\wedge\cdots\wedge d''x_r,\]
	\[d'\eta' = \left(\sum\limits_{i=1}^r(-1)^{i-1}\dfrac{\partial \eta_i'}{\partial x_i}\right)d'x_1\wedge\cdots\wedge d'x_r\wedge d''x_1\wedge\cdots\wedge d''x_r,\] 
	\[ \int_\sigma d'\eta'=(-1)^{r(r-1)/2}\int_{\varphi(\sigma)} \sum\limits_{i=1}^r(-1)^{i-1}\dfrac{\partial \eta_i'}{\partial x_i}dx_1\wedge\cdots\wedge dx_r = (-1)^{r(r-1)/2}\int_{\varphi(\sigma)} d\omega,\]	
	where $\omega =\sum\limits_{i=1}^r\eta_{i}'dx_1\wedge\cdots\wedge \widehat{dx_i}\wedge\cdots\wedge dx_r$ and $dx_i$ are viewed as smooth forms on $\R^r$. 
	For every facet $\tau \leq \sigma$, we take a basis $e'_1, \cdots, e'_{r-1}$ of $N_\tau$ such that $\{e_1',\cdots, e_{r-1}', \omega_{\sigma,\tau}\}$ form a basis of $N$. Moreover, we assume that $x_1,\cdots, x_r$ are the corresponding coordinate functions. Then we have \begin{align*}
		(-1)^r\int_\tau\langle \eta';(\emptyset,(\omega_{\sigma,\tau})_{\{1\}})\rangle&=- \int_\tau\eta_r'd'x_1\wedge\cdots\wedge d'x_{r-1}\wedge d''x_1\wedge\cdots\wedge d''x_{r-1} \\&=-(-1)^{(r-1)(r-2)/2}\int_{\varphi(\tau)}\eta_r'dx_1\wedge\cdots\wedge dx_{r-1}\\& = (-1)^{r(r-1)/2}\cdot (-1)^r\int_{\varphi(\tau)}\omega,
	\end{align*}
	notice that here $(-1)^r\int_{\varphi(\tau)}\omega$ is the integration of $\omega$ on $\varphi(\tau)$ with boundary orientation, see \cite[Example~22.13]{tu2008an}. Although our boundary $\partial\sigma$ is not smooth, it is piecewise smooth, then we have the Stokes' formula for simplexes, see 
	\cite[4.7]{warner1983foundations}. So the equalities hold.
\end{proof}

\begin{proposition}[\cite{chambert2012formes}~Lemma~1.3.8~Green's formula]
	Keep the notion in \cref{stokes' formula for polyhedra}. For any symmetric $(p,p)$-superform $\alpha$, symmetric $(q,q)$-superform $\beta$ around $|C|$ with $p+q=r-1$ and $\Supp(\alpha)\cap\Supp(\beta)$ compact, we have
	\[\int_{C}\alpha\wedge d'd''\beta-d'd''\alpha\wedge \beta = \int_{\partial C}\alpha\wedge d''\beta-d''\alpha\wedge\beta.\]
	\label{Green's formula for superforms}
\end{proposition}


\subsection{Supercurrents}

\label{supercurrents on N}




\begin{definition}
	Let $\widetilde{\Omega}\subset N_\R$ be an open subset.
	A {\bf supercurrent of type $(p,q)$} on $\widetilde{\Omega}$ is a continuous $\R$-linear map $T: A^{d-p,d-q}_c(\widetilde{\Omega})\rightarrow \R$. The space of all supercurrents of type $(p,q)$ is denoted by ${D}^{p,q}(\widetilde{\Omega})$, and ${D}^n(\widetilde{\Omega}): = \bigoplus\limits_{p+q=n}{D}^{p,q}(\widetilde{\Omega})$, ${D}(\widetilde{\Omega}): = \bigoplus\limits_{n=1}^r{D}^{n}(\widetilde{\Omega})$.
	
	The space ${D}(\widetilde{\Omega})$ form a sheaf of bigraded $\R$-vector spaces on $N_\R$. 
	\label{definition of supercurrents}
\end{definition}
\begin{remark}
	\label{remark4.4}
	\begin{enumerate}
		
		\item[(1)] If $\widetilde{\Omega}'\subset \widetilde{\Omega}$, then $A^{p,q}_c(\widetilde{\Omega}')\subset A^{p,q}_c(\widetilde{\Omega})$.  So we have the restriction map $D^{p,q}(\widetilde{\Omega}) \rightarrow D^{p,q}(\widetilde{\Omega}')$. It is well-known that $D^{p,q}(\widetilde{\Omega})$ form a fine sheaf on $N_\R$. 
		\item[(2)] By definition, for any $T\in D^{p,q}(\widetilde{\Omega})$, we have
		\begin{align*}
			\widetilde{\Omega}\setminus\Supp(T)& = \{x\in \widetilde{\Omega}\mid T|_{U_x}=0 \text{ for some open neighborhood $U_x$ of $x$ in $\widetilde{\Omega}$}\}\\
			&=\left\{x\in \widetilde{\Omega}\,\middle\vert\, \parbox[c]{.6\linewidth}{ there is an open neighborhood $U_x$ of $x$ in $\widetilde{\Omega}$ such that $T(\eta)=0$ for any $\eta\in A_c^{r-p,r-q}(U_x)$}\right\}.
		\end{align*}
		This implies that $\Supp(T)$ is the minimal closed subset $V$ of $\widetilde{\Omega}$ such that $\eta|_V =0$ (in an open neighborhood of $V$) implies $T(\eta)=0$. In particular, $T(\eta)=T(\eta')$ if $\eta|_{\Supp(T)}=\eta'|_{\Supp(T)}$. With this fact, we can generalize $T(\eta)$ to those $\eta\in A^{r-p,r-q}(\widetilde{\Omega})$ with $\Supp(T)\cap \Supp(\eta)$ compact by setting $T(\eta)=T(\eta')$ for any $\eta'\in A^{r-p,r-q}_c(\widetilde{\Omega})$ such that $\eta|_{\Supp(T)}=\eta'|_{\Supp(T)}$, this value is independent of the choice of $\eta'$. 
		\begin{proof}
			The equalities hold from the definition of $\Supp(T)$. We will show the implication.
			Assume that $\eta\in A^{r-p, r-q}(\widetilde{\Omega})$ such that $\eta|_U=0$ on an open neighborhood U of $\Supp(T)$. By partition of unity, we can write $\eta=\sum\limits_{i=0}^\infty\eta_i$ such that for any $i$, there is $x\in \widetilde{\Omega}\setminus\Supp(T)$ such that $\Supp(\eta_i)\subset U_x$, where $U_x$ is an open neighborhood of $x$ satisfying the property in the description of $\widetilde{\Omega}\setminus\Supp(T)$. Hence $T(\eta_i)=0$ and $T(\eta)=0$. On the other hand, if $V\subset \widetilde{\Omega}$ is a closed subset with this property, then we can see that $\widetilde{\Omega}\setminus V\subset\widetilde{\Omega}\setminus\Supp(T)$ from the equalities of sets.	
		\end{proof}
	\end{enumerate}

\end{remark}



\begin{definition}
	Let $\widetilde{\Omega}\subset N_\R$ be an open subset. For $T\in {D}^{p,q}(\widetilde{\Omega})$ and $\alpha\in A^{p',q'}(\widetilde{\Omega})$, the {\bf wedge product} 
	$T\wedge \alpha\in {D}^{p+p',q+q'}(\widetilde{\Omega})$ is defined by
	\[(T\wedge \alpha)(\gamma): = T(\alpha\wedge\gamma)\]
	for any $\gamma\in A^{r-p-p',r-q-q'}_c(\widetilde{\Omega})$.
	We denote $\alpha\wedge T:=(-1)^{(p+q)(p'+q')} T\wedge\alpha$.
	\label{definition of wedge product}
\end{definition}

We have the following important examples of supercurrents.

\begin{example}
	Let $\widetilde{\Omega}\subset N_\R$ be an open subset.
	\begin{enumerate}
		\item [(1)] For any $\omega\in L_{\mathrm{loc}}^{p,q}(\widetilde{\Omega})$, i.e. a locally integrable $(p,q)$-form, we have the map 
		\begin{align*}
			[\omega]: A_c^{d-p,d-q}(\widetilde{\Omega})&\rightarrow \R,\\
			\eta&\mapsto \int_{\widetilde{\Omega}}\omega\wedge\eta
		\end{align*}
		is a supercurrent of type $(p,q)$. Hence we define an $\R$-linear map 
		\[[\cdot]: L_{\mathrm{loc}}^{p,q}(\widetilde{\Omega})\rightarrow D^{p,q}(\widetilde{\Omega}).\]
		Moreover, the restriction $[\cdot]: A^{p,q}(\widetilde{\Omega})\rightarrow D^{p,q}(\widetilde{\Omega})$ is injective, and for any $\alpha,\beta\in A(\widetilde{\Omega})$, we have \[[\alpha]\wedge\beta = [\alpha\wedge\beta] = \alpha\wedge[\beta].\]
		\item[(2)] Let $\sigma\subset N_\R$ be an integral $\R$-affine polyhedron of codimension $l$, i.e. dimension $r-l$. Then a {\bf supercurrent of Dirac type} $\delta_\sigma\in D^{l,l}(\widetilde{\Omega})$ arising from $\sigma$ is given by
		\begin{align*}
			\delta_\sigma: A_c^{r-l, r
				-l}(\widetilde{\Omega})&\rightarrow \R,\\
			\alpha&\mapsto \int_{\sigma}\alpha: = \int_{\sigma}i^*\alpha,
		\end{align*}
		where $i: N_{\sigma,\R}\rightarrow N_\R$ is the natural embedding, and the right-hand side is given in \cref{definition of integration}. 
		
		
	\end{enumerate}
\label{example of supercurrents}
\end{example}
\begin{remark}
	\begin{enumerate}
		\item [(1)] For convenience and without confusion, we may write the image of $\omega\in L_{\mathrm{loc}}^{p,q}(\widetilde{\Omega})$ in $D^{p,q}(\widetilde{\Omega})$ as $\omega$ or $[\omega]$, and the image of $C\in F^l(N_\R)$ in $D^{l,l}(\widetilde{\Omega})$ as $C, [C]$ or $\delta_C$.
	\end{enumerate}
\label{remark of example of supercurrents}
\end{remark}

Similarly, we have the following operators.

\begin{lemma}
	Let $\widetilde{\Omega}\subset N_\R$ be an open subset. We can define the following maps between supercurrents. The maps
	\begin{align*}
		d': {D}^{p,q}(\widetilde{\Omega})\rightarrow {D}^{p+1,q}(\widetilde{\Omega}),\\
		d'': {D}^{p,q}(\widetilde{\Omega})\rightarrow {D}^{p,q+1}(\widetilde{\Omega}),\\
		J: {D}^{p,q}(\widetilde{\Omega})\rightarrow {D}^{q,p}(\widetilde{\Omega}),
	\end{align*}
	are defined as follows: for any $T\in {D}^{p,q}(\widetilde{\Omega})$,
	\begin{align*}
		d'(T): {A}_c^{r-p-1, r-q}(\widetilde{\Omega})& \rightarrow \R,\\
		\gamma&\mapsto (-1)^{p+q+1}T(d'\gamma);\\
		d''(T): {A}_c^{r-p, r-q-1}(\widetilde{\Omega}) &\rightarrow \R,\\
		\gamma&\mapsto (-1)^{p+q+1}T(d''\gamma);\\
		J(T):  A_c^{r-q,r-p}(\widetilde{\Omega}) & \rightarrow \R, \\
		\gamma&\mapsto (-1)^rT(J\gamma).
	\end{align*}
	Then the following hold.
	\begin{enumerate}
		\item [(1)] For any superform $\alpha, \beta\in A(\widetilde{\Omega})$, we have 
		\[d'[\alpha] = [d'\alpha], \ \ d''[\alpha] = [d''\alpha], \ \ J[\alpha] = [J\alpha].\]
		
		\item[(2)] For any $T\in {D}^{p,q}(\widetilde{\Omega})$ and $\alpha\in A^{p',q'}(\widetilde{\Omega}),$
		\[J(T\wedge\alpha) = JT\wedge J\alpha,\]
		\[d'(T\wedge \alpha) = d'T\wedge \alpha+ (-1)^{p+q}T\wedge d'\alpha,\]
		It is similar for $d''$.
		\item[(3)] We have 
		\[d'd'=d''d''=0,\]
		\[d'd''=-d''d'.\]
		\item[(4)] 	We have	\[{JJ}=\id,\]
		\[d'={J}d''{J},\ \ {J}d'=d''{J},\]
		\[d''={J}d'{J}, \ \ {J}d"=d'{J}.\]
	\end{enumerate}
	We say that a supercurrent $T\in D^{p,p}(\widetilde{\Omega})$ is {\bf symmetric} if $JT=(-1)^{p}T$.
\end{lemma}
	

\begin{example}
	Let $\widetilde{\Omega}\subset N_\R$ be an open subset. Let $\sigma\subset N_\R$ be an integral $\R$-affine polyhedron. Then $\delta_\sigma\in D^{l,l}(\widetilde{\Omega})$ is symmetric. 
\end{example}

\begin{definition}
	Let $F: N'_\R\rightarrow N_\R$ be an integral $\R$-affine map, and $\widetilde{\Omega}\subset N_\R$, $\widetilde{\Omega}'\subset N_\R'$ open subsets such that $F(\widetilde{\Omega}')\subset \widetilde{\Omega}$. Let $T\in D^{p,q}(\widetilde{\Omega}')$. Assume that $\Supp(T)\cap F^{-1}(K)$ is compact for any compact subset $K\subset \widetilde{\Omega}$ (i.e. $F|_{\Supp(T)}: \Supp(T)\rightarrow \widetilde{\Omega}$ is compact). The {\bf direct image $F_*T \in D^{r-r'+p,r-r'+q}(\widetilde{\Omega})$ of $T$} by $F$ is given by 
	\[F_*T(\eta): = T(F^*\eta)\]
	for any $\eta\in A^{r'-p,r'-q}_c(\widetilde{\Omega})$. This is well-defined by \cref{remark4.4}~(3).
	\label{direct image of current}
\end{definition}
\begin{remark}
	\begin{enumerate}
		\item [(1)] If $\Supp(T)$ is compact, then $F_*T$ is always defined.
		\item[(2)] The map $F_*$ is $\R$-linear for suitable supercurrents, and
		\[d'F_*=F_*d', \ \ d''F_*=F_*d''.\]
	\end{enumerate}
\end{remark}


\begin{example}
	Keep the notion in \cref{direct image of current}. Let $\sigma'$ be a polyhedron in $N_\R'$. Assume that $\dim\sigma'=\dim F(\sigma')$, then we have
	\[F_*\delta_{\sigma'} = \delta_{\widehat{F}_*\sigma'},\]
	where $\widehat{F}_*$ is the classical push-forward of a polyhedron in tropical geometry.
\end{example}
\begin{proof}
	This is from \cref{integral via bijection} and the definition of $F_*$, i.e. $F_*([\sigma]) = [N: dF(N')][F(\sigma)].$
\end{proof}
\begin{remark}
	\begin{enumerate}
		\item [(1)] For general case, this is not true, since $\widehat{F}_*\sigma'=0$ if $\dim\sigma'<\dim F(\sigma')$, see \cref{push-forward of polyhedral supercurrents} for general case.
	\end{enumerate}
\end{remark}


\section{Polyhedral supercurrents on $N_\R$}

\label{polyhedral supercurrents on N}

In this section, we will recall polyhedral supercurrents given in \cite[Section~2]{gubler2017a}, they are equivalence classes of integral $\R$-affine polyhedral complexes with smooth superforms. We will study integration of polyhedral supercurrents. As before, throughout this section, $N$ denotes a finitely generated free abelian group of rank $r$.

\subsection{Polyhedral supercurrents}

Let $\sigma\subset N_{\R}$ be an integral $\R$-affine polyhedron. For any open subset $\Omega\subset N_{\sigma,\R}$, we write $A^{p,q}_{N_{\sigma,\R}}(\Omega)$ for the set of $(p,q)$-superforms on $\Omega$. For an open subset $\widetilde{\Omega}\subset N_\R$, we set
 \[A^{p,q}_\sigma(\widetilde{\Omega}\cap\sigma) := i^*A^{p,q}(\widetilde{\Omega}) \subset A_{N_{\sigma,\R}}^{p,q}(\widetilde{\Omega}\cap\mathrm{relint}(\sigma)),\]
where $i: \widetilde{\Omega}\cap\mathrm{relint}(\sigma) \hookrightarrow \widetilde{\Omega}$.
Equivalently, $A^{p,q}_\sigma(\widetilde{\Omega}\cap\sigma)$ is the restrictions of $(p,q)$-superforms in $A^{p,q}_{N_{\sigma,\R}}(\widetilde{\Omega}\cap N_{\sigma,\R})$ on $\widetilde{\Omega}\cap\mathrm{relint}(\sigma)$.
\begin{remark}
	\begin{enumerate}
		\item [(1)] Notice that an element $\alpha\in A_{N_{\sigma,\R}}^{p,q}(\widetilde{\Omega}\cap\mathrm{relint}(\sigma))$ is not necessarily induced from $A(\widetilde{\Omega})$ unless the map $\widetilde{\Omega}\cap\mathrm{relint}(\sigma)\hookrightarrow \widetilde{\Omega}$ is a proper embedding, see \cite[Lemma~5.34]{lee2012introduction}. On the other hand, by partition of unity, $A^{p,q}_\sigma(\widetilde{\Omega}\cap\sigma)$ is the restrictions of $(p,q)$-forms in $A^{p,q}_{N_{\sigma,\R}}(N_{\sigma,\R})$ on $\widetilde{\Omega}\cap\mathrm{relint}(\sigma)$. 
		\item[(2)] If $\tau\leq \sigma$, then we have the restriction map
		\begin{align*}
			A^{p,q}_\sigma(\widetilde{\Omega}\cap\sigma)& \rightarrow A^{p,q}_\tau(\widetilde{\Omega}\cap\tau),\\
			\alpha&\mapsto \alpha|_\tau,
		\end{align*}
		where $\alpha|_\tau = \beta|_{\widetilde{\Omega}\cap\mathrm{relint}(\tau)}$ if $\alpha = \beta|_{\widetilde{\Omega}\cap\mathrm{relint}(\sigma)}$ with $\beta\in A^{p,q}(\widetilde{\Omega})$. This is only dependent on $\alpha$, $\widetilde{\Omega}\cap \sigma$, and is independent of the choice of $\widetilde{\Omega}$ or $\beta$.
	\end{enumerate}
\end{remark}

\begin{example}
	Let $\widetilde{\Omega}\subset N_\R$ be an open subset. Let $\sigma\subset N_\R$ be an integral $\R$-affine polyhedron of codimension $l$, i.e. dimension $r-l$. For any $\alpha\in A^{p,q}_\sigma(\widetilde{\Omega}\cap\sigma)$, we define $\alpha\wedge\delta_\sigma\in D^{p+l, q+l}(\widetilde{\Omega})$ as follows:
	\begin{align*}
		\alpha\wedge\delta_\sigma: A_c^{r-p-l,r-q-l}(\widetilde{\Omega})& \rightarrow \R,\\
		\beta&\mapsto \langle\delta_\sigma, \alpha\wedge i^*\beta\rangle = \int_{\sigma}\alpha\wedge i^*\beta,
	\end{align*}
	where $i: N_{\sigma,\R}\hookrightarrow N_\R$ is the natural embedding. 
\end{example}

\begin{definition}
	Let $\widetilde{\Omega}$ be an open subset of $N_\R$.  A supercurrent $\alpha\in D(\widetilde{\Omega})$ is called {\bf polyhedral} if there exists an integral $\R$-affine polyhedral decomposition $\mathcal{C}$ of $N_\R$ and a family $(\alpha_\sigma)_{\sigma\in \mathcal{C}}$ of superforms $\alpha_\sigma\in A_\sigma(\widetilde{\Omega}\cap\sigma)$ such that
	\[\alpha=\sum\limits_{\sigma\in \mathcal{C}}\alpha_\sigma\wedge \delta_\sigma\]
	in $D(\widetilde{\Omega})$. In this case we say that $\mathcal{C}$ is {\bf adapted to $\alpha$}.
	
	We denote $P^{p,q,l}(\widetilde{\Omega})\subset D^{p+l,q+l}(\widetilde{\Omega})$ the space of polyhedral supercurrents of the form $\sum\limits_{\sigma\in \mathcal{C}^l}\alpha_\sigma\wedge \delta_\sigma$ with $\alpha_\sigma\in A_\sigma^{p,q}(\widetilde{\Omega}\cap\sigma)$. Such element is called a {\bf $(p,q,l)$-polyhedral supercurrent}. We set 
	\[P^{p,q}(\widetilde{\Omega}) := \bigoplus\limits_{l\geq 0}P^{p-l,q-l,l}(\widetilde{\Omega}), \ \ P^n(\widetilde{\Omega}): = \bigoplus\limits_{p+q=n}P^{p,q}(\widetilde{\Omega}), \ \ P(\widetilde{\Omega}): = \bigoplus\limits_{n\geq 0}P^n(\widetilde{\Omega}).\]
	Obviously, we have $A^{p,q}(\widetilde{\Omega})\subset P^{p,q}(\widetilde{\Omega})$.  We set $P^{p,q}_{n}(\widetilde{\Omega}):=P^{p,q,r-n}(\widetilde{\Omega})$.
	\label{definition of polyhedral supercurrents}
\end{definition}
\begin{remark}\label{rk:definition of polyhedral supercurrents}
	\begin{enumerate}
		\item[(1)] For simplicity, we take $\mathcal{C}$ to be a polyhedral decomposition of $N_\R$, but the definition is the same if $\mathcal{C}$ runs through all polyhedral complexes. 
		\item[(2)] We see that $(\alpha_\sigma)_{\sigma\in \mathcal{C}}$ is uniquely determined by $\alpha$ and $\mathcal{C}$, and 
		\[\Supp(\alpha) = \bigcup\limits_{\sigma\in\mathcal{C}}\Supp(\alpha_\sigma).\]
		\begin{proof}
			For the uniqueness, we should show that $\alpha_\sigma=\beta_\sigma$ for any $\sigma\in \mathcal{C}^l$ if $\sum\limits_{\sigma\in \mathcal{C}^l}\alpha_\sigma\wedge\delta_\sigma=\sum\limits_{\sigma\in \mathcal{C}^l}\beta_\sigma\wedge\delta_\sigma\in P^{p,q,l}(\widetilde{\Omega})$. This comes from the fact that the map \[[\cdot]: A^{p,q}_{N_{\sigma,\R}}(\widetilde{\Omega}\cap \mathrm{relint}(\sigma))\rightarrow D^{p,q}(\widetilde{\Omega}\cap \mathrm{relint}(\sigma))\] is injective.
		\end{proof}

		\item[(3)] For any open subsets $U\subset\widetilde{\Omega} \subset N_\R$, and $\alpha=\sum\limits_{\sigma\in \mathcal{C}}\alpha_\sigma\wedge \delta_\sigma \in D(\widetilde{\Omega})$, we have 
		\[\alpha|_U = \sum\limits_{\sigma\in \mathcal{C}}\alpha_\sigma|_{U\cap\mathrm{relint}(\sigma)}\wedge \delta_\sigma \in D(U),\]
		where  $\alpha_\sigma|_{U\cap\mathrm{relint}(\sigma)} \in A_\sigma(U\cap\sigma)$ is induced by an element in $A(U)$. Notice that, $P^{p,q,l}$ is a separated presheaf and $d', d''$ are local. It is not a sheaf. However, for any finite open covering $\widetilde{\Omega}=\bigcup\limits_{i=1}^n\widetilde{\Omega}_i$, the sequence 
		\[\xymatrix{0\ar[r] & P^{p,q}(\widetilde{\Omega}) \ar[r] & \prod\limits_{i=1}^nP^{p,q}(\widetilde{\Omega}_i) \ar[r] &\prod\limits_{i,j=1}^nP^{p,q}(\widetilde{\Omega}_i\cap\widetilde{\Omega}_j) }\]
		is exact.
	\end{enumerate}
\end{remark}

\subsection{Polyhedral complexes with smooth superforms}

For further discussion, we give a tropical way to describe polyhedral supercurrents.

\begin{definition}
	Let $\widetilde{\Omega}$ be an open subset of $N_\R$. An integral $\R$-affine polyhedral complex $\mathcal{C}$ of pure dimension $n$ is said to be {\bf with smooth $(p,q)$-superforms on $\widetilde{\Omega}$} if each polyhedral $\sigma\in\mathcal{C}_n$ is given with a smooth $(p,q)$-superform $\alpha_\sigma\in A^{p,q}_\sigma(\widetilde{\Omega}\cap\sigma)$. Such complex is denoted by $(\mathcal{C}_n, \{\alpha_{\sigma}\}_{\sigma\in \mathcal{C}_n})$. The support of $(\mathcal{C}_n, \{\alpha_{\sigma}\}_{\sigma\in \mathcal{C}_n})$ is given as
	\[|(\mathcal{C}_n, \{\alpha_{\sigma}\}_{\sigma\in \mathcal{C}_n})|:=\bigcup\limits_{\sigma\in \mathcal{C}_n}\Supp(\alpha_\sigma)\]
	
	We say two integral $\R$-affine polyhedral complexes with smooth superforms $(\mathcal{C}_n, \{\alpha_{\sigma}\}_{\sigma\in \mathcal{C}_n})$, $(\mathcal{C}_n', \{\alpha_{\sigma'}'\}_{\sigma'\in \mathcal{C}_n}')$ of pure dimension $n$ are {\bf equivalent} if $|(\mathcal{C}_n, \{\alpha_{\sigma}\}_{\sigma\in \mathcal{C}_n})| = |(\mathcal{C}_n', \{\alpha_{\sigma'}'\}_{\sigma'\in \mathcal{C}_n'})|$ and $\alpha_\sigma = \alpha'_{\sigma'}$ in $\widetilde{\Omega}\cap\mathrm{relint}(\sigma\cap \sigma')$ when $\sigma\cap\sigma'$ is of dimension $n$ for all $\sigma\in \mathcal{C}$ and $\sigma'\in \mathcal{C}'$.  The set of equivalence classes is denoted by $F_n^{p,q}(\widetilde{\Omega})$, which is an abelian group, and the image of $(\mathcal{C}_n, \{\alpha_{\sigma}\}_{\sigma\in \mathcal{C}_n})$ is denoted by $\sum\limits_{\sigma\in\mathcal{C}_n}\alpha_\sigma[\sigma]$. We set $F^{p,q,l}(\widetilde{\Omega}):=F_{r-l}^{p,q}(\widetilde{\Omega})$.
	\label{polyhedral complexes with smooth forms}
\end{definition}
\begin{remark}\label{rk:polyhedral complexes with smooth forms}
	\begin{enumerate}
		\item [(1)] Inspired by \cite[Section~6.1]{chambert2021topics}, we can view $F_n^{p,q}(\widetilde{\Omega})$ as an $\R$-vector space  generated by $\alpha_\sigma[\sigma]$ with $\dim\sigma\leq n$ and $\alpha_\sigma\in A^{p,q}_\sigma(\widetilde{\Omega}\cap\sigma)$ and satisfying the following relations:
		\begin{itemize}
			\item $a\cdot(\alpha_\sigma[\sigma]) = (a\alpha_\sigma)[\sigma]$ for any $a\in \R$ and $\alpha_\sigma\in A^{p,q}_\sigma(\widetilde{\Omega}\cap\sigma)$.
			\item $\alpha_\sigma[\sigma]=0$ for every polyhedron $\sigma$ such that $\dim\sigma<n$;
			\item $\alpha_\sigma[\sigma]+\alpha_\sigma|_{\sigma\cap H}[\sigma\cap H]=\alpha_\sigma|_{\sigma\cap V_+}[\sigma\cap V_+]+ \alpha_\sigma|_{\sigma\cap V_-}[\sigma\cap V_-]$ where $\sigma$ is an $n$-polyhedron in $N_\R$, $\alpha_\sigma\in A^{p,q}_\sigma(\widetilde{\Omega}\cap\sigma)$ and $V_+, V_-$ are half-spaces such that $V_+\cap V_-=H$ is a hyperplane and $N_\R=V_+\cup V_-$.
		\end{itemize}
		For the last relation, if $\sigma\subset H$, then the equality is trivial; if $\sigma\not\subset H$, then $\alpha_\sigma|_{\sigma\cap H}[\sigma\cap H]=0$.
		
		\item[(2)] An $\R$-linear map from $F_n^{p,q}(\widetilde{\Omega})$ to a given $\R$-vector space $W$ is given by a map 
		\[\lambda: \{\alpha_\sigma[\sigma]\mid \dim\sigma\leq n \text{ and } \alpha_\sigma\in A^{p,q}_\sigma(\widetilde{\Omega}\cap\sigma)\}\rightarrow W\]
		satisfying the following relation
		\begin{itemize}
			\item $\lambda((a\cdot\alpha_\sigma)[\sigma]) = a\lambda(\alpha_\sigma[\sigma])$ for any $a\in \R$ and $\alpha_\sigma\in A^{p,q}_\sigma(\widetilde{\Omega}\cap\sigma)$.
			\item $\lambda(\alpha_\sigma[\sigma])=0$ for every polyhedron $\sigma$ such that $\dim\sigma<n$;
			\item $\lambda(\alpha_\sigma[\sigma])+\lambda(\alpha_\sigma|_{\sigma\cap H}[\sigma\cap H])=\lambda(\alpha_\sigma|_{\sigma\cap V_+}[\sigma\cap V_+])+ \lambda(\alpha_\sigma|_{\sigma\cap V_-}[\sigma\cap V_-])$ for every hyperplane $H$ of $N_\R$ dividing $N_\R$ into two closed half-spaces $V_+, V_-$.
		\end{itemize}
		\item[(3)] Notice that $F_n^{p,q}(\widetilde{\Omega}$) has a natural $C^\infty(\widetilde{\Omega})$-module structure: for any $f\in C^\infty(\widetilde{\Omega})$ and $\alpha_\sigma\in A^{p,q}_\sigma(\widetilde{\Omega}\cap\sigma)$, \[f\cdot(\alpha_\sigma[\sigma]) := (f|_{\widetilde{\Omega}\cap\mathrm{relint}(\sigma)}\cdot\alpha_\sigma)[\sigma].\]
	\end{enumerate}
\label{remark:polyhedralcomplexwithsmoothforms}
\end{remark}

\begin{proposition}
	Let $\widetilde{\Omega}\subset N_\R$ be an open subset. We have an isomorphism of $C^\infty(\widetilde{\Omega})$-modules 
	\begin{align*}
		F^{p,q,l}(\widetilde{\Omega})&\overset{\sim}{\rightarrow} P^{p,q,l}(\widetilde{\Omega}),\\
		\alpha_\sigma[\sigma]&\mapsto \alpha_\sigma\wedge\delta_\sigma.
	\end{align*}
	\label{equivalence of polyhedral supercurrent and polyhedral complex with smooth superforms}
\end{proposition}
\begin{proof}
	The map is well-defined since it satisfies the relations in \cref{rk:polyhedral complexes with smooth forms}~(2). It remains to show it is injective. This is given in \cref{rk:definition of polyhedral supercurrents}~(2).
\end{proof}

\ 

We will always identify $F^{p,q,l}(\widetilde{\Omega})$ and $P^{p,q,l}(\widetilde{\Omega})$. We will not use the notation $F^l(\widetilde{\Omega})$ for $\bigoplus\limits_{p,q}F^{p-l,q-l,l}(\widetilde{\Omega})$ since $F^l(\widetilde{\Omega})= F^{0,0,l}(\widetilde{\Omega})$ from our notations.

\begin{definition}
	Let $\widetilde{\Omega}$ be an open subset of $N_\R$, and $\alpha = \sum\limits_{\sigma\in \mathcal{C}}\alpha_\sigma[\sigma]$ a polyhedral supercurrent. The {\bf polyhedral derivatives} $d'_P(\alpha)$ and $d''_P(\alpha)$ are defined as
	\[d_P'(\alpha): = \sum\limits_{\sigma\in \mathcal{C}}d'(\alpha_\sigma)[\sigma],\]
	\[d_P''(\alpha): = \sum\limits_{\sigma\in \mathcal{C}}d''(\alpha_\sigma)[\sigma],\]
	which are independent of the choice of the polyhedral complex of definition $\mathcal{C}$.
\end{definition} 
\begin{remark}
	\begin{enumerate}
		\item [(1)] $d'_P$ and $d_P''$ are well-defined since they satisfy the relation list in \cref{rk:polyhedral complexes with smooth forms}~(2).
		\item[(2)] The polyhedral derivatives of a polyhedral supercurrent $\alpha$ do not coincide with derivative of $\alpha$ in sense of supercurrents. 
		\item[(3)] The Leibniz rules hold for polyhedral derivatives $d_P', d_P''$.
	\end{enumerate}
\end{remark}
\subsection{Integration of polyhedral supercurrents}

We will define integration of polyhedron supercurrents, which generalizes the one of superforms, i.e. \cref{definition of integration}. 

\begin{definition}
	Let $\widetilde{\Omega}$ be an open subset of $N_\R$. For any $\alpha = \sum\limits_{\sigma\in \mathcal{C}}\alpha_\sigma[\sigma] \in P^{r,r}_c(\widetilde{\Omega})$, we set  
	\[\int_{\widetilde{\Omega}}\alpha: = \sum\limits_{\sigma\in\mathcal{C}}\int_{\sigma\cap \widetilde{\Omega}}\alpha_\sigma.\]
	This defines an $\R$-linear map $P^{r,r}_c(\widetilde{\Omega})\rightarrow \R$.
	
	For any integral $\R$-affine polyhedral set $P\subset \widetilde{\Omega}$ (can be of any dimension), we set 
	\[\int_{P}\alpha: = \sum\limits_{\sigma\in\mathcal{C}}\int_{P\cap\sigma\cap \widetilde{\Omega}}\alpha_\sigma,\] 
	called the {\bf integral of $\alpha$ over $P$}. This is independent of the choice of $\mathcal{C}$. 
	\label{definition of integration for polyhedral supercurrents}
\end{definition}
\begin{remark}
	\begin{enumerate}
		\item [(1)] This definition generalizes the integration on $A_c^{r,r}(\widetilde{\Omega})$ on a polyhedron. 
		\item [(2)] We have an observation: for any $\alpha\in P_c^{r,r}(\widetilde{\Omega})$, we can take any function $f\in A_c^{0,0}(\widetilde{\Omega})$ such that $f|_{\Supp(\alpha)} \equiv 1$. Then
		\[\alpha(f) = \int_{\widetilde{\Omega}}\alpha.\]
		This equality can be seen as the integral of general supercurrent.
		\item[(3)] Note that $\int_P\alpha$ is well-defined for any $\alpha\in F^{r,r}(\widetilde{\Omega})$ with $\Supp(\alpha)\cap P$ compact. 
		\item[(4)] We may assume that $\alpha$ admits a polyhedral complex of definition $\mathcal{C}$ such that $P$ has a polyhedral decomposition $\mathcal{D}$ which is a subcomplex of $\mathcal{C}$. In this case, we write
		\[\int_P\alpha := \sum\limits_{\sigma\in \mathcal{D}}\int_\sigma\alpha_\sigma.\]
	\end{enumerate}
\end{remark}

\begin{definition}
	Let $\widetilde{\Omega}$ be an open subset of $N_\R$, $P\subset \widetilde{\Omega}$ an integral $\R$-affine polyhedral set,  and $\eta=\sum\limits_{\sigma\in \mathcal{C}}\eta_\sigma[\sigma]\in P^{r-1,r}(\widetilde{\Omega})$ a polyhedral supercurrent with $\Supp(\eta)\cap P$ compact. As above, $\eta$ admits a polyhedral complex of definition $\mathcal{C}$ such that $P$ has a polyhedral decomposition $\mathcal{D}$ which is a subcomplex of $\mathcal{C}$. We define the {\bf integral of $\eta$ over the boundary} $P$ as 
		\[\int_{\partial P}\eta := \sum\limits_{\sigma\in \mathcal{D}}\int_{\partial\sigma}\eta_\sigma,\]
	the right-hand side is defined in \cref{definition of integration on boundary for superforms}. For a polyhedral supercurrent $\eta \in P^{r,r-1}(\widetilde{\Omega})$, the definition $\int_{\partial P}\eta$ is similar.
\end{definition}

From Stokes' formula, i.e. \cref{stokes' formula for polyhedra}, we can easily deduce the Stokes' formula for polyhedral sets. 

\begin{proposition}[Stokes' formula for polyhedral supercurrents]
	Let $\widetilde{\Omega}$ be an open subset of $N_\R$, and $P\subset \widetilde{\Omega}$ an integral $\R$-affine polyhedral set. Then we have
	\[\int_Pd'_P\alpha = \int_{\partial P}\alpha, \ \ \int_Pd''_P\beta = \int_{\partial P}\beta,\]
	for all polyhedral supercurrents $\alpha\in P^{r-1,r}(\widetilde{\Omega})$ and $\beta\in P^{r,r-1}(\widetilde{\Omega})$ with $\Supp(\alpha)\cap P$ and $\Supp(\beta)\cap P$ compact.
	\label{stokes' formula for polyhedral supercurrents}
\end{proposition}





\section{$\delta$-Forms on $N_\R$}

\label{delta forms on N}

Mihatsch \cite{mihatsch2021on} defines and studies $\delta$-forms on $\R^n$. His definition is natural and contains the $\delta$-preforms defined in \cite[Definition~2.9]{gubler2017a}. In this section, at first we will recall his main result in \cite{mihatsch2021on}, then we will define corner locus and prove the tropical Poincar\'e-Lelong formula.


As before, throughout this section, $N, N'$ are finitely generated free abelian group of rank $r, r'$ respectively, and we fix an open subset $\widetilde{\Omega}\subset N_\R$.




\subsection{Piecewise linear and piecewise smooth superforms}

\begin{definition}
	A {\bf piecewise smooth superform}  on $\widetilde{\Omega}$ is a set $\alpha=\{\alpha_\sigma\}_{\sigma\in \mathcal{C}}$, where $\mathcal{C}$ is an integral $\R$-affine polyhedral decomposition of $N_\R$ and $\alpha_\sigma \in A_\sigma(\widetilde{\Omega}\cap \sigma)\subset A_{N_{\sigma,\R}}(\widetilde{\Omega}\cap \mathrm{relint}(\sigma))$, such that $\alpha_\tau=\alpha_\sigma|_\tau$ for any $\tau\leq \sigma\in \mathcal{C}$. A {\bf piecewise linear superform} on $\widetilde{\Omega}$ is a piecewise smooth superform $\alpha=\{\alpha_\sigma\}_{\sigma\in \mathcal{C}}$ such that each $\alpha_{\sigma}$ is given of the form $\sum\limits_{I,J}\alpha_{\sigma,IJ}d'x_I\wedge d''x_J$ with $\alpha_{\sigma,IJ}$ restriction of a linear function $N_\R\rightarrow \R$.
	
	We identify two superforms $\alpha, \alpha'$ on $\widetilde{\Omega}$ if they have the same support and if $\alpha_\sigma =\alpha'_{\sigma'}$ on $\widetilde{\Omega}\cap\mathrm{relint}(\sigma)\cap\mathrm{relint}(\sigma')$ for all polyhedra $\sigma\in \mathcal{C}$ and $\sigma'\in \mathcal{C}'$.
	
	The space of piecewise smooth superforms (resp. piecewise linear superforms) on $\widetilde{\Omega}$ is denoted by $PS(\widetilde{\Omega})$ (resp. $PL(\widetilde{\Omega})$). It has a natural bigrading $C^{\infty}(\widetilde{\Omega})$-algebra with $\wedge$-product. We denote by $PS_c(\widetilde{\Omega})\subset PS(\widetilde{\Omega})$ the subspace of piecewise smooth superforms with compact support.
	
	We have differentials $d_P', d_P''$ on $PS(\widetilde{\Omega})$ defined as
	\[d_P'\alpha = \{d'\alpha_\sigma\}_{\sigma\in \mathcal{C}}, \ \ d_P''\alpha = \{d''\alpha_\sigma\}_{\sigma\in \mathcal{C}},\]
	for any $\alpha=\{\alpha_\sigma\}_{\sigma\in \mathcal{C}}$.
	
	Obviously, an elements $\{\phi_\sigma\}_{\sigma\in \mathcal{C}} \in PS^{0,0}(\Omega)$ (resp. $PL^{0,0}(\widetilde{\Omega}$)) gives a function $\widetilde{\Omega} \rightarrow \R$, called a {\bf piecewise smooth function} (resp. a {\bf piecewise linear function}). 
	\label{piecewise smooth superform on open subset}
\end{definition}
\begin{remark}
	\begin{enumerate}
		\item[(1)] For an open subset $\widetilde{\Omega}\subset N_\R$, we have $A(\widetilde{\Omega})\hookrightarrow PS(\widetilde{\Omega})$ and 
		\[PS^{p,q}(\widetilde{\Omega})\hookrightarrow P^{p,q,0}(\widetilde{\Omega}), \ \ (\alpha_\sigma)_{\sigma\in \mathcal{C}} \mapsto \sum\limits_{\sigma\in\mathcal{C}^0}\alpha_\sigma[\sigma].\]
		\item[(2)] As usual, for $\alpha\in PS^k(\widetilde{\Omega})$ and $\beta\in PS(\widetilde{\Omega})$, we have
		\[d_P'(\alpha\wedge\beta) = d_P'\alpha\wedge\beta+(-1)^k\alpha\wedge d_P'\beta.\]
		An analogous formula holds for $d_P''$.
		
	\end{enumerate}
\end{remark}

\subsection{$\delta$-forms and balancing condition}

\begin{definition}
 A polyhedral supercurrent $\alpha\in B(\widetilde{\Omega})$ is called a {\bf $\delta$-form} if $d'\alpha$ and $d''\alpha$ are again polyhedral.  
	The {\bf support} of a $\delta$-form is the support of its underlying supercurrent.
	
	The subspace of $\delta$-forms in $P^{p,q,l}(\widetilde{\Omega})$ is denoted by $B^{p,q,l}(\widetilde{\Omega})$. We set 
	\[B^{p,q}(\widetilde{\Omega}) := \bigoplus\limits_{l\geq 0}B^{p-l,q-l,l}(\widetilde{\Omega}), \ \ B^n(\widetilde{\Omega}): = \bigoplus\limits_{p+q=n}B^{p,q}(\widetilde{\Omega}), \ \ B(\widetilde{\Omega}): = \bigoplus\limits_{n\in \N}B^n(\widetilde{\Omega}).\]
	We denote by $B_c(\widetilde{\Omega})$ the subspace of $B(\widetilde{\Omega})$ given by the $\delta$-forms with compact support.
	
	\label{delta form on an open subset}
\end{definition}
\begin{remark}
	\begin{enumerate}
		\item[(1)] Notice that, comparing to \cite[Example~2.2]{mihatsch2021on}, in \cref{definition of polyhedral supercurrents}, our polyhedral complex $\mathcal{C}$ is integral $\R$-affine, so we have a natural calibrage  \cite{chambert2012formes} or weight in \cite[Definition~2.1]{mihatsch2021on} for each polyhedron in $\mathcal{C}$. 
		\item[(2)] The $\R$-vector spaces $B^{p,q,l}(\widetilde{\Omega})$ form a separated presheaf on $N_\R$ since $P^{p,q,l}$ is a separated presheaf and $d', d''$ are local. It is not a sheaf. However, for finite open covering $\widetilde{\Omega}=\bigcup\limits_{i=1}^n\widetilde{\Omega}_i$, the sequence 
		\[\xymatrix{0\ar[r] & B^{p,q}(\widetilde{\Omega}) \ar[r] & \prod\limits_{i=1}^nB^{p,q}(\widetilde{\Omega}_i) \ar[r] &\prod\limits_{i,j=1}^nB^{p,q}(\widetilde{\Omega}_i\cap\widetilde{\Omega}_j) }\]
		is exact.
	\end{enumerate}
\end{remark}

Since $P^{p,q,l}(\widetilde{\Omega})\simeq F^{p,q,l}(\widetilde{\Omega})$, we introduce the balance condition to give another description for $\delta$-forms.

\begin{definition}
 A polyhedral supercurrent $\alpha=\sum\limits_{\sigma\in\mathcal{C}^l} \alpha_\sigma[\sigma]\in P^{p,q,l}(\widetilde{\Omega})\simeq F^{p,q,l}(\widetilde{\Omega})$ is called {\bf balanced} or satisfies the {\bf balancing condition} if for any $\tau\in \mathcal{C}^{l+1}$, the sum
	\[\sum\limits_{\substack{\sigma\in \mathcal{C}^l\\ \tau\leq \sigma \text{ a facet}}}\alpha_{\sigma}|_\tau\otimes\omega_{\sigma,\tau}\in A_\tau^{p,q}(\widetilde{\Omega}\cap\tau)\otimes_\R N_\R\]
	lies in the subspace $A_\tau^{p,q}(\widetilde{\Omega}\cap\tau)\otimes_\R N_{\tau,\R}$, where $\omega_{\sigma,\tau}\in N_\sigma$ is a normal vector (obviously independent of the choice of normal vectors). A polyhedral supercurrent $\alpha = \sum\limits_{p,q,l}\alpha^{p,q,l}\in P(\widetilde{\Omega}) = \bigoplus\limits_{p,q,l}P^{p,q,l}(\widetilde{\Omega})$ is called {\bf balanced} if each component $\alpha^{p,q,l}$ is so.
	
	Notice that the isomorphism $F^l(N_\R)\simeq P^{0,0,l}(N_\R)$ identifies $TZ^l(N_\R)$ in \cite[Definition~1.1~(v)]{gubler2017a} with the set of balanced $(0,0,l)$-polyhedral supercurrents on $N_\R$, so we denote $TZ^l(\widetilde{\Omega})$ the set  of balanced $(0,0,l)$-polyhedral supercurrents on $\widetilde{\Omega}$.
	\label{balancing condition for differential forms}
\end{definition} 
\begin{remark}
	\begin{enumerate}
		\item [(1)] The definition is independent of the choice of representatives of $\alpha$, i.e. if $\alpha=\sum\limits_{\sigma\in\mathcal{C}^l} \alpha_\sigma[\sigma]= \sum\limits_{\sigma\in\mathcal{D}^l} \beta_\sigma[\sigma]$, then  $(\alpha_\sigma)_{\sigma\in \mathcal{C}}$ satisfies the balancing condition if and only if $(\beta_{\sigma})_{\sigma\in \mathcal{D}}$ satisfying the balancing condition.
		 
		\begin{proof}
			We can assume that $\mathcal{D}$ is a refinement of $\mathcal{C}$ and $\mathcal{C}, \mathcal{D}$ are of codimension $l$. If $\rho\in \mathcal{D}^{l+1}$, let $\tau\in \mathcal{C}$ such that $\rho\subset \tau$. If $\dim\tau = \dim\rho$, then there is a bijection 
			\[\{\varrho\in \mathcal{D}^l\mid \rho\subset \varrho\}\longleftrightarrow \{\sigma\in \mathcal{C}_n\mid \tau\subset \sigma\}.\]
			Then balancing condition at $\tau$ implies the balancing condition at $\rho$. If $\dim\tau = \dim\rho+1$, then the plane $H$ given by $\rho$ divides $N_\tau$ into two half-space $V_+$, $V_-$ with $H=V_+\cap V_-$. In this case, $\{\varrho\in \mathcal{D}^l\mid \rho\subset \varrho\}$ is consisted of two polyhedra $\varrho_+, \varrho_-$ lying in $V_+$ and $V_-$ respectively. We can take $\omega_{\varrho_+,\rho} = - \omega_{\varrho_-,\rho}$, then $\beta_{\varrho_+}|_{\rho}\otimes\omega_{\varrho_+,\rho}+\beta_{\varrho_-}|_{\rho}\otimes\omega_{\varrho_-,\rho} = 0$. Hence, the balancing condition for $(\alpha_\sigma)_{\sigma\in \mathcal{C}}$ implies the one for $(\beta_\sigma)_{\sigma\in \mathcal{D}}$.
			
			The converse statement comes from the fact that $\tau\in\mathcal{C}^{l+1}$ is a union of some elements in $\mathcal{D}^{l+1}$, then we apply the statement above.
		\end{proof}
		\item[(2)] The definition is equivalent with \cite[Definition~3.2]{mihatsch2021on}, i.e. a polyhedral supercurrent $\alpha=\sum\limits_{\sigma\in\mathcal{C}^l} \alpha_\sigma[\sigma]\in P(\widetilde{\Omega})$ if and only if for any $\tau\in \mathcal{C}$, the sum
		\[\sum\limits_{\substack{\sigma\in \mathcal{C}\\ \tau\leq \sigma \text{ a facet}}}\alpha_{\sigma}|_\tau\otimes\omega_{\sigma,\tau}\in A_\tau(\widetilde{\Omega}\cap\tau)\otimes_\R N_\R\]
		lies in the subspace $A_\tau^{p,q}(\widetilde{\Omega}\cap\tau)\otimes_\R N_{\tau,\R}$.
	\end{enumerate}
\end{remark}

One of the main result in \cite{mihatsch2021on} is the following theorem.

\begin{theorem}[\cite{mihatsch2021on}~Theorem~3.3]
 Then a polyhedral supercurrent $\alpha$ is a $\delta$-form if and only if it is balanced. In particular, $\alpha$ is a $\delta$-form if and only if its $(p,q,l)$-component $\alpha^{p,q,l}$ is a $\delta$-form for all $(p,q,l)$.
	
	Furthermore, if one out of $d'\alpha$, $d''\alpha$ is polyhedral, then $\alpha$ is a $\delta$-form.
	\label{main theorem of delta forms}
\end{theorem}

\begin{corollary}
	The derivatives
	\[d': B^{p,q}(\widetilde{\Omega})\rightarrow B^{p+1,q}(\widetilde{\Omega}), \ \ d': B^{p,q}(\widetilde{\Omega})\rightarrow B^{p,q+1}(\widetilde{\Omega}),\]
	and polyhedral derivatives 
	\[d'_P: B^{p,q,l}(\widetilde{\Omega})\rightarrow B^{p+1,q,l}(\widetilde{\Omega}), \ \ d'_P: B^{p,q,l}(\widetilde{\Omega})\rightarrow B^{p,q+1,l}(\widetilde{\Omega}),\]
	are well-defined.
\end{corollary}
\begin{proof}
	If $\alpha\in B^{p,q}(\widetilde{\Omega})$, then $d'd'\alpha=0$, which is polyhedral, so $d'\alpha\in B^{p+1,q}(\widetilde{\Omega})$ by \cref{main theorem of delta forms}. Similarly, $d''\alpha\in B^{p,q+1}(\widetilde{\Omega})$.
	
	As for $d_P'\alpha$ and $d_P''\alpha$, it is easy to check the balancing condition if $\alpha$ is balanced.
\end{proof}

\begin{proposition}
	We have a natural $\R$-linear isomorphism $PS^{p,q}(\widetilde{\Omega})\simeq B^{p,q,0}(\widetilde{\Omega})$. Moreover, this map identifies $d', d''$ on $PS^{p,q}(\widetilde{\Omega})$ with $d'_P, d''_P$ on $B^{p,q,0}(\widetilde{\Omega})$ respectively.
\end{proposition}
\begin{proof}
	For any $\alpha=\{\alpha_{\sigma}\}_{\sigma\in\mathcal{C}}\in PS^{p,q}(\widetilde{\Omega})$, the polyhedral supercurrent $\sum\limits_{\sigma\in\mathcal{C}^0}\alpha_\sigma[\sigma]\in P^{p,q,0}(\widetilde{\Omega})$ satisfies the balancing condition. Indeed, for any $\tau\in \mathcal{C}^{1}$, there are only two polyhedra $\sigma_1,\sigma_2$ of codimension $0$ containing $\tau$ as a face. Let $\omega\in N$ be a representative of the generator of $N/N_\tau$ pointing $\sigma_1$, then $-\omega$ pointing $\sigma_2$. Then we have 
	\[\alpha_{\sigma_1}|_\tau\otimes\omega + \alpha_{\sigma_2}|_\tau\otimes(-\omega) =0 \in A^{p,q}_\tau(\tau\cap\widetilde{\Omega})\otimes N_{\tau,\R}.\]
	
	Conversely,  given $\omega$ as above, notice that $\R\omega\cap N_{\tau,\R}=0$, then we must have $\alpha_{\sigma_1}|_\tau=\alpha_{\sigma_2}|_\tau$ for any $\tau =\sigma_1\cap\sigma_2$. Hence we can set $\alpha_\rho :=\alpha_{\sigma}|_\rho$ for any face of $\sigma\in\mathcal{C}^0$, which is independent of the choice of $\sigma$. Then $\{\alpha_\sigma\}_{\sigma\in \mathcal{C}}\in PS^{p,q}(\widetilde{\Omega})$. 
\end{proof}

\begin{definition}
	From the proof of \cref{main theorem of delta forms} in \cite{mihatsch2021on}, we can define the {\bf boundary derivatives} $\partial':= d'-d'_P$ and $\partial'':=d''-d''_P$,
	\[\partial': B^{p,q,l}(\widetilde{\Omega})\rightarrow B^{p,q-1,l+1}(\widetilde{\Omega}), \ \ \partial'': B^{p,q,l}(\widetilde{\Omega})\rightarrow B^{p-1,q,l+1}(\widetilde{\Omega})\]
\end{definition}
\begin{remark}
	\begin{enumerate}
		\item[(1)] Since for a form $\alpha\wedge \beta$ with $\alpha\in A(\widetilde{\Omega}), \beta\in B(\widetilde{\Omega})$, we have the Leibniz rules with respect to $d', d'', d'_P, d''_P$, so we also have Leibniz rule with respect to $\partial', \partial''$. After giving the definition of $\wedge$-product, we will have Leibniz rules with respect to all derivatives.
	\end{enumerate}
\end{remark}

\begin{corollary}[\cite{gubler2017a}~Proposition~2.18]
	We have $d'=d'_P: B^{0,0,l}(\widetilde{\Omega})\rightarrow B^{1,0,l}(\widetilde{\Omega})$, i.e. $\partial'=\partial''=0$. It is similar for $d'', d_P''$ and $\partial''$. Moreover, these properties hold for $\delta$-preforms $A^{p,q}(\widetilde{\Omega})\wedge B^{0,0,l}(\widetilde{\Omega})$, also see \cite[Definition~2.9]{gubler2017a}. 
	\label{boundary derivatives are null on delta preforms}
\end{corollary}
\begin{proof}
	This is from the definition of $\partial', \partial''$ and the Leibniz rules. 
\end{proof}

\begin{lemma}
	The boundary derivatives satisfy
	\begin{align*}
		0&=\partial'\partial'=\partial''\partial'',\\
		0&=\partial'\partial''+\partial''\partial',\\
		0&=\partial' d_P'+d_P'\partial'=\partial'' d_P''+d_P''\partial'',\\
		0&=\partial' d_P''+d_P'\partial''+\partial'' d_P'+d_P''\partial'.
	\end{align*}
\end{lemma}

Mihatsch calculation $\partial'\alpha$ in \cite{mihatsch2021on}, for readers' convenience and further application, we give his result here. 

\begin{proposition}[\cite{mihatsch2021on}~Section~3]
	Let $\alpha=\sum\limits_{\sigma\in \mathcal{C}^l}\alpha_\sigma[\sigma]\in B^{p,q,l}(\widetilde{\Omega})$. Then $\partial'\alpha = (-1)^{r+p+q+1}\sum\limits_{\tau\in \mathcal{C}^{l+1}}\beta_\tau[\tau]\in B^{p,q-1,l+1}(\widetilde{\Omega})$, where $\beta_\tau$ can be given in the following way: write \[\sum\limits_{\substack{\sigma\in \mathcal{C}^l\\ \tau< \sigma \text{ a facet}}}\alpha_{\sigma}|_\tau\otimes\omega_{\sigma,\tau} = \sum\limits_{i\in I}\beta_i\otimes v_i\in A_\tau^{p,q}(\widetilde{\Omega}\cap\tau)\otimes_\R N_{\tau,\R}.\]
	Then 
	\[\beta_\tau = \sum\limits_{\substack{\sigma\in\mathcal{C}^l\\\tau< \sigma \text{ a facet}}}\langle\alpha_\sigma;(\emptyset,(\omega_{\sigma,\tau})_{\{1\}})\rangle|_\tau-\sum\limits_{i\in I}\langle \beta_i; (\emptyset, (v_i)_{\{1\}})\rangle.\]
	\label{expression of boundary derivatives}
\end{proposition}
	

\subsection{$\wedge$-product and tropical Poincar\'e-Lelong formula}

To define general intersection or $\wedge$ in $B(\widetilde{\Omega})$, let us consider the simplest case, i.e. defining a bilinear map $PS^{p,q}\times B^{p',q',l'}\rightarrow B^{p+p',q+q',l'}$. 

\begin{definition}
	Let $\alpha\in PS(\widetilde{\Omega})$ and $\beta\in B(\widetilde{\Omega})$. Let $\mathcal{C}$ be a polyhedral decomposition of $N_\R$ that adapts $\alpha$ and $\beta$, say
	\[\alpha=(\alpha_\sigma)_{\sigma\in\mathcal{C}}=\sum\limits_{\sigma\in\mathcal{C}^0}\alpha_\sigma[\sigma], \ \ \beta=\sum\limits_{\sigma\in\mathcal{C}}\beta_\sigma[\sigma].\]
	We define their product as
	\[\alpha\wedge \beta :=\sum\limits_{\sigma\in\mathcal{C}}\alpha_\sigma\wedge\beta_\sigma[\sigma].\]
	\label{wedge product for ps and delta-form}
\end{definition}
\begin{remark}
	\begin{enumerate}
		\item [(1)] We see that the definition is independent of the choice of $\mathcal{C}$, and we have $\alpha\wedge \beta\in B(\widetilde{\Omega})$ by the balancing condition. Moreover, it defines a bilinear map \[PS^{p,q}(\widetilde{\Omega})\times B^{p',q',l'}(\widetilde{\Omega})\rightarrow B^{p+p',q+q',l'}(\widetilde{\Omega}).\]
		\item[(2)] For any $\alpha\in PS^{p,q}(\widetilde{\Omega}), \beta\in B^{p',q'}(\widetilde{\Omega})$ and $\gamma\in A_c^{r-p-p',r-q-q'}(\widetilde{\Omega})$, we have 
		\[(\alpha\wedge\beta)(\gamma) = \sum\limits_{\sigma}\int_{\widetilde{\Omega}\cap\sigma}\alpha_\sigma\wedge\beta_\sigma\wedge\gamma = (-1)^{(p+q)(p'+q')}\beta(\alpha\wedge\gamma),\]
		so this definition generalizes \cref{definition of wedge product}.
	\end{enumerate}
\end{remark}

\begin{definition}
	Let $\alpha=\sum\limits_{\sigma\in\mathcal{C}^l}\alpha_\sigma[\sigma]\in B^{p,q,l}(\widetilde{\Omega})$, and $\phi=(\phi_\sigma)_{\sigma\in\mathcal{C}}\in PS^{0,0}(\widetilde{\Omega})$ with polyhedral complex of definition $\mathcal{C}$.  For any $\tau\in \mathcal{C}^{l+1}$, we write
	\[\sum\limits_{\substack{\sigma\in \mathcal{C}^l\\ \tau< \sigma \text{ a facet}}}\alpha_{\sigma}|_\tau\otimes\omega_{\sigma,\tau} = \sum\limits_{i\in I}\beta_i\otimes v_i\in A_\tau^{p,q}(\widetilde{\Omega}\cap\tau)\otimes_\R N_{\tau,\R},\]
	where $\omega_{\sigma,\tau}$ is a normal vector for $\sigma/\tau$, and the value of $\alpha_\sigma$ on $\tau$ is given by continuous extension. Notice that $\omega_\tau$ depends on the choice of $\omega_{\sigma,\tau}$. We set
	\[\alpha_\tau:=\sum\limits_{\substack{\sigma\in\mathcal{C}^l\\\tau< \sigma \text{ a facet}}}\dfrac{\partial\phi_\sigma}{\partial \omega_{\sigma,\tau}}\alpha_\sigma|_\tau-\sum\limits_{i\in I}\dfrac{\partial\phi_\tau}{\partial v_i}\beta_i.\]
	The {\bf corner locus} of $\phi$ for $\alpha$ is defined as $\mathrm{div}(\phi)\cdot \alpha:= \sum\limits_{\tau\in\mathcal{C}^{l+1}}\alpha_\tau[\tau]\in P^{p,q,l+1}(\widetilde{\Omega})$, which is $\R$-bilinear at $\phi$ and $\alpha$. We will show that $\mathrm{div}(\phi)\cdot \alpha\in B^{p,q,l+1}(\widetilde{\Omega})$, see \cref{tropical poincare-lelong formula}.
	\label{corner locus}
\end{definition}
\begin{remark}
	\begin{enumerate}
		\item [(1)] The corner locus $\mathrm{div}(\phi)\cdot C$ for $\phi\in PL^{0,0}(N_\R)$, $C \in TZ^l(N_\R)$ with integral weights is defined in \cite{allermann2010first} and \cite[Definition~1.10]{gubler2017a}, called Weil divisor of $\phi$ on $C$. 
		We use the notion $\mathrm{div}(\phi)\cdot C$ rather than $\phi\cdot C$. 
		\item[(2)] It is not hard to see that $\mathrm{div}(\phi)\cdot \alpha$ is independent of the choice of $\mathcal{C}$, and is only dependent of the function $\phi|_{\Supp(\alpha)}$ for a given $\alpha \in B^{p,q,l}(\widetilde{\Omega})$. 
	\end{enumerate}
\end{remark}

\begin{example}
	Let $H\subset N_\R$ be an integral $\R$-affine space. Let $\phi=\max\{\phi_1, \phi_2\}: N_\R\rightarrow \R$ be a piecewise linear function, i.e. $\phi_1, \phi_2: N_\R\rightarrow \R$ are integral $\R$-affine functions. Set $V_\phi:=\{x\in N_\R\mid \phi_1(x)=\phi_2(x)\}$. Then
	\[\mathrm{div}(\phi)\cdot [H]=\begin{cases}
		0 & \text{ if $H\subset V_\phi$;}\\
		\dfrac{\partial(\phi_1-\phi_2)}{\partial\omega_{V_1,V_\phi}}[N:N_H+N_\phi]\cdot[V_\phi\cap H]& \text{ if $H\not\subset V_\phi$,}
	\end{cases}\]
	where $V_1=\{x\in N_\R\mid \phi_1(x)\geq\phi_2(x)\}$, $N_\phi=N\cap V_\phi$ and $\omega_{V_1,V_\phi}$ is a normal vector over $V_1/V_\phi$. Moreover, let $\psi=\min\{\phi_1,\phi_2\}$, then $\mathrm{div}(\psi)\cdot[H] = -\mathrm{div}(\phi)\cdot[H]$.
\end{example}
\begin{proof}
	Notice that for a given $H$, the corner locus $\mathrm{div}(\phi)\cdot [H]$ is only dependent on $\phi|_H: H\rightarrow \R$, so we can assume that $H=N_\R$. If $H\subset V_\phi$, i.e. $N_\R=V_\phi$ then $\phi$ is an integral $\R$-affine function. We take $\mathcal{C}=\{N_\R\}$ which is not consisted of lower faces, then $\mathrm{div}(\phi)=0$. If $H\not\subset V_\phi$, we consider half-spaces $V_1, V_2$ divided by $V_\phi$ such that $\phi|_{V_1}=\phi_1$ and $\phi|_{V_2}=\phi_2$. We can take $\omega_{V_2,V_\phi} = -\omega_{V_1,V_\phi}$, then $\mathrm{div}(\phi)\cdot [N_\R] = m[V_\phi]$ with
	\[m = \dfrac{\partial(\phi_1-\phi_2)}{\partial\omega_{V_1,V_\phi}}.\]
	For the second case, in general $H\subset N_\R$, we have  $\mathrm{div}(\phi)\cdot [H] = \dfrac{\partial(\phi_1-\phi_2)}{\partial\omega_{H_1,H_\phi}}[H_\phi]$, where $H_1=V_\cap H_\phi$ and $H_\phi=V_\phi\cap H$. We consider the relation between $\omega_{H_1,H_\phi}$ and $\omega_{V_1,V_\phi}$. Indeed, 
	\[\Z\overline{\omega_{H_1, H_\phi}}\simeq N_H/(N_H\cap N_\phi)\hookrightarrow N/N_\phi\simeq \Z\overline{\omega_{V_1,V_\phi}}\]
	and both sides are of rank $1$, so $\overline{\omega_{H_1, H_\phi}} = [N: N_H+N_\phi]\overline{\omega_{V_1,V_\phi}}$, this implies the first statement.
	
	For the second one, assume that $H\not\subset V_\phi$, then  
	\begin{align*}
		\mathrm{div}(\psi)\cdot[H] &= -\mathrm{div}(-\psi)\cdot[H]\\ &=-\dfrac{\partial(-\phi_1-(-\phi_2))}{\partial(-\omega_{V_1,V_\phi})}[N:N_H+N_\phi]\cdot[V_\phi\cap H]\\&=-\mathrm{div}(\phi)\cdot[H].
	\end{align*}
\end{proof}
\begin{remark}
	\begin{enumerate}
		\item [(1)] Let $\phi, \psi$ be as in the example. Then for any $\alpha\in B^{p,q,l}(\widetilde{\Omega})$, by \cref{basic property of corner locus} which we will prove late, we have $\mathrm{div}(\psi)\cdot\alpha = -\mathrm{div}(\phi)\cdot\alpha$. This is from the fact
		that  $\mathrm{div}(\psi)\cdot\alpha= \mathrm{div}(\psi)\cdot([N_\R]\wedge\alpha) = (\mathrm{div}(\psi)\cdot[N_\R])\wedge\alpha$.
	\end{enumerate}
\end{remark}

With the result above, we give a representation of graph of a integral $\R$-affine map.

\begin{example}
	Let $F: N'_R\rightarrow N_\R$ be a morphism in $\mathcal{AT}op_{\Z,\R}$, and $(x_1',\cdots, x_{r'}')$, $(x_1,\cdots, x_{r})$ coordinate functions of $N_\R'$ and $N_\R$ respectively. Let $\Delta_F\subset N_\R'\times N_\R$ be the graph of $F$. Then we have 
	\[[\Delta_F]=\mathrm{div}(\phi_1)\cdots\mathrm{div}(\phi_r)\cdot [N_\R'\times N_\R],\]
	where $\phi_i=\max\{f_i(x_1',\cdots, x_{r'}'), x_i\}$ and $F=(f_1,\cdots, f_r): N_\R'\rightarrow N_\R$.
\end{example}
\begin{proof}
	Set \[V_i=\{(x',x)\in N_\R\times N_\R\mid f_i(x')\geq x_i\},\]\[V_{\phi_i}=\{(x',x)\in N_\R\times N_\R\mid f_i(x')= x_i\}.\] 
	On $V_{\phi_i}$, the only non-free variable is $x_i$, then $(0,e_i)$ is a normal vector for $V_i/V_{\phi_i}$, where $e_i\in N_\R$ is the vector corresponding $x_i$. Notice $M+(N'\times N)\cap V_{\phi_i}=N'\times N$ for any subgroup $M\subset N'\times N$ containing $(0, e_i)$. Since $(0,e_i)\in V_{\phi_{i-1}}\cap\cdots\cap V_{\phi_r}$ for any $1\leq i\leq r$, we have
	\begin{align*}
		\mathrm{div}(\phi_1)\cdots\mathrm{div}(\phi_r)\cdot [N_\R'\times N_\R]&=\mathrm{div}(\phi_1)\cdots\mathrm{div}(\phi_{r-1})\cdot [V_{\phi_{r}}]\\&\cdots\\
		&=[V_{\phi_1}\cap\cdots\cap V_{\phi_r}]\\
		&=[\Delta_F].
	\end{align*}
	
\end{proof}





\begin{proposition}
	Let $N, N'$ be finitely generated abelian group, and $\widetilde{\Omega}\subset N_\R, \widetilde{\Omega}'\subset N_\R'$ open subsets. For homogeneous supercurrents $T\in D^{n}(\widetilde{\Omega}), T'\in D^{n'}(\widetilde{\Omega}')$, we have a unique supercurrent $T\times T'\in D^{n+n'}(\widetilde{\Omega}\times\widetilde{\Omega}')$ such that
	\[(T\times T')(p_1^*\eta\wedge p_2^*\eta') = (-1)^{nn'}T(\eta)T'(\eta')\]
	for any $\eta\in A(\widetilde{\Omega})$ and $\eta'\in A(\widetilde{\Omega}')$,
	where $p_1: N_\R\times N'_\R\rightarrow N_\R$, $p_2: N_\R\times N'_\R\rightarrow N'_\R$ are projections. Moreover, the following hold.
	\begin{enumerate}
		\item [(1)] $d'(T\times T') = d'T\times T'+(-1)^nT\times d'T'$, and it is similar for $d''$.
		\item[(2)] If  $T = \alpha_\sigma[\sigma]$ and $T'=\alpha'_{\sigma'}[\sigma']$, then 
		\[T\times T'=(p_1^*\alpha_{\sigma}\wedge p_2^*\alpha_{\sigma'})[\sigma\times \sigma'].\]
		\item[(3)] If $T\in B^{p,q,l}(\widetilde{\Omega})$ and $T'\in B^{p',q',l'}(\widetilde{\Omega}')$, then $T\times T'\in B^{p+p',q+q',l+l'}(\widetilde{\Omega}\times \widetilde{\Omega})$, and 
		\[d'_P(T\times T') = d'_PT\times T'+(-1)^nT\times d'_PT',\] it is similar for $d''_P$, $\partial', \partial''$.
	\end{enumerate}
	\label{product of currents}
\end{proposition}
\begin{proof}
	The proof is based on the fact set $p_1^*A(\widetilde{\Omega})\wedge p_2^*A(\widetilde{\Omega}')$ is dense in $A(\widetilde{\Omega}\times\widetilde{\Omega}')$.  See \cite[Section~1.2]{demailly2012complex}, and the properties are given in \cite[Section~4]{mihatsch2021on}.
\end{proof}

\begin{theorem}[\cite{mihatsch2021on}~Theorem~4.1]
	There is a unique way to define an associative bilinear product $\wedge: B(\widetilde{\Omega})\times B(\widetilde{\Omega})\rightarrow B(\widetilde{\Omega})$ that satisfies the following properties:
	\begin{enumerate}
		\item[(1)] the Leibniz rules with respect to $d'$ and $d''$: for any $\alpha\in B(\widetilde{\Omega}')$ and $\beta\in B(\widetilde{\Omega})$, we have 
		\[d'(\alpha\wedge \beta) = d'\alpha\wedge \beta + (-1)^{\deg\alpha}\alpha\wedge d'\beta,\]
		and it is similar for $d''$;
		\item[(2)] for any $\alpha\in PS(\widetilde{\Omega})$  and $\beta\in B(\widetilde{\Omega})$, the $\wedge$-product $\alpha\wedge \beta$ is given in \cref{wedge product for ps and delta-form};
		\item[(3)] for any $\alpha,\beta\in B(\widetilde{\Omega})$, we have $\Supp([\Delta]\wedge(\alpha\times \beta))\subset\Delta$ and
		\[\alpha\wedge \beta = p_{1,*}([\Delta]\wedge(\alpha\times \beta)),\]
		where $\Delta\subset N_\R\times N_\R$ is the diagonal and $p_1: N_\R\times N_\R\rightarrow N_\R$ is the first projection.  
	\end{enumerate}
	\label{intersection}
\end{theorem}
\begin{proof}
	We give a sketch of the proof. 
	
	Uniqueness. If such map exists, 
	then for any $\phi\in PS^{0,0}(\widetilde{\Omega})$ and $\beta\in B(\widetilde{\Omega})$, we have $d'\phi = d'_P\phi$, $d''\phi = d_P''\phi$ which are in $PS^1(\widetilde{\Omega})$, and the following important equality
	\begin{align}
		(d'd''\phi)\wedge \beta= \partial'd''_P\phi\wedge\beta+ d'_Pd''_P\phi\wedge\beta
		=\partial'(d_P''\phi\wedge\beta)+d_P''\phi\wedge\partial'\beta+d_P'd_P''\phi\wedge\beta.
		\label{wedge product1}
	\end{align}
	The final one holds since the Leibniz rule holds with respect to $d'$ and $d''_P$ for $PS^{p,q}(\widetilde{\Omega})\wedge B^{p,q}(\widetilde{\Omega})$, so holds with respect to $\partial'$. 
	This implies that $(d'd''\phi)\wedge \beta$ is uniquely determined by properties (1) (2).
	
	We fix integral coordinate functions $x_1,\cdots, x_{r}$ and $y_1,\cdots, y_r$ on $N_\R\times N_\R$. Then we have 
	\begin{align*}
		[\Delta]=d'd''\phi_1\wedge\cdots \wedge d'd''\phi_r,
		\label{differential representation of diagonal}
	\end{align*}
	where $\phi_i=\max\{x_i, y_i\}$. This is from the tropical Poincar\'e-Lelong formula, \cref{tropical poincare-lelong formula}. We will prove it below, and notice that the proof only use equality~(\ref{wedge product1}). 
	
	Since $\Supp([\Delta]\wedge(\alpha\times \beta))\subset\Delta$, we know that $p_{1,*}([\Delta]\wedge(\alpha\times \beta))$ is well-defined. This implies that $\alpha\wedge\beta$ is uniquely determined by (3).
	
	For existence, from the discussion above, we can define $[\Delta]\wedge(\alpha\times\beta)=d'd''\phi_1\wedge\cdots \wedge d'd''\phi_r\wedge(\alpha\times\beta)$ by equality~(\ref{wedge product1}) and property~(2), then the general case will be given from (3).
\end{proof}
\begin{remark}
	\begin{enumerate}
		\item [(1)] Obviously, $\wedge$-product generalizes the stable intersection in $F(N_\R)$ and $\wedge$-product in $A(\widetilde{\Omega})$ to $B(\widetilde{\Omega})$. For any $\alpha, \beta\in B(\widetilde{\Omega})$, we have 
		\[\Supp(\alpha\wedge\beta)\subset \Supp(\alpha)\cap\Supp(\beta).\]
		
	\end{enumerate}
\end{remark}

The following result is called the tropical Poincar\'e-Lelong formula, which generalizes \cite[Corollary~3.19]{gubler2017a}.

\begin{theorem}[The tropical Poincar\'e-Lelong formula]
	For $\alpha\in B^{p,q,l}(\widetilde{\Omega})$ and $\phi\in PS^{0,0}(\widetilde{\Omega})$. Then we have 
	\[\mathrm{div}(\phi)\cdot \alpha =  d'd''\phi\wedge\alpha-d_P'd_P''\phi\wedge\alpha= \partial'(d_P''\phi\wedge\alpha)+d_P''\phi\wedge\partial'\alpha\]
	in $P(\widetilde{\Omega})$. In particular, $\mathrm{div}(\phi)\cdot \alpha\in B^{p,q,l+1}(\widetilde{\Omega})$.
	\label{tropical poincare-lelong formula}
\end{theorem}
\begin{proof}
	From equality~(\ref{wedge product1}), it is sufficient to show that 
	\[\mathrm{div}(\phi)\cdot \alpha = \partial'(d_P''\phi\wedge\alpha)+d_P''\phi\wedge\partial'\alpha.\]
	Assume that $\phi = \sum\limits_{\sigma\in\mathcal{C}^0}\phi_\sigma[\sigma]$ and $\alpha=\sum\limits_{\sigma\in\mathcal{C}^l}\alpha_\sigma[\sigma]\in B^{p,q,l}(\widetilde{\Omega})$. Then 
	$d_P''\phi\wedge\alpha= \sum\limits_{\sigma\in\mathcal{C}^l}(d''\phi_\sigma\wedge\alpha_\sigma)[\sigma]$. Write 
	\[\sum\limits_{\substack{\sigma\in \mathcal{C}^l\\ \tau< \sigma \text{ a facet}}}\alpha_{\sigma}|_\tau\otimes\omega_{\sigma,\tau} = \sum\limits_{i\in I}\beta_i\otimes v_i\in A_\tau^{p,q}(\widetilde{\Omega}\cap\tau)\otimes_\R N_{\tau,\R},\]
	then by \cref{expression of boundary derivatives} we have
	\begin{align*}
		\partial'(d_P''\phi\wedge\alpha) = \sum\limits_{\tau\in\mathcal{C}^{l+1}} (\sum\limits_{\substack{\sigma\in\mathcal{C}^l\\\tau< \sigma \text{ a facet}}}\langle d''\phi_\sigma\wedge\alpha_\sigma;(\emptyset,(\omega_{\sigma,\tau})_{\{1\}})\rangle|_\tau-\sum\limits_{i\in I}\langle d''\phi_\tau\wedge\beta_i; (\emptyset, (v_i)_{\{1\}})\rangle)[\tau],
	\end{align*}
	\begin{align*}
		d_P''\phi\wedge\partial'\alpha = \sum\limits_{\tau\in\mathcal{C}^{l+1}} (d''\phi_\tau\wedge\sum\limits_{\substack{\sigma\in\mathcal{C}^l\\\tau< \sigma \text{ a facet}}}\langle\alpha_\sigma;(\emptyset,(\omega_{\sigma,\tau})_{\{1\}})\rangle|_\tau-\sum\limits_{i\in I}\langle \beta_i; (\emptyset, (v_i)_{\{1\}})\rangle)[\tau].
	\end{align*}
	By \cref{basic properties of contraction}, we have
	\begin{align*}
		\partial'(d_P''\phi\wedge\alpha)+d_P''\phi\wedge\partial'\alpha= \sum\limits_{\tau\in\mathcal{C}^{l+1}} (\sum\limits_{\substack{\sigma\in\mathcal{C}^l\\\tau< \sigma \text{ a facet}}}\dfrac{\partial\phi_\sigma}{\partial \omega_{\sigma,\tau}}\alpha_\sigma|_\tau-\sum\limits_{i\in I}\dfrac{\partial\phi_\tau}{\partial v_i}\beta_i)[\tau] = \mathrm{div}(\phi)\cdot\alpha.
	\end{align*}
\end{proof}
\begin{remark} 
	\begin{enumerate}
		\item [(1)] In particular, if $\phi\in PL^{0,0}(\widetilde{\Omega})$, or more general, $\phi$ is piecewise linear on $\Supp(\alpha)$, then 
		\[\mathrm{div}(\phi)\cdot \alpha = d'd''\phi\wedge \alpha.\]
	\end{enumerate}
\end{remark}

From the tropical Poincare-Lelong formula, we have the following corollary.

\begin{corollary}
	For $\alpha\in B^{p,q,l}(\widetilde{\Omega})$ and $\phi\in PS^{0,0}(\widetilde{\Omega})$.
	\begin{enumerate}
		\item [(1)] (Associativity law) For any $\beta\in B^{p',q',l'}(\widetilde{\Omega})$, we have
		\[\mathrm{div}(\phi)\cdot(\alpha\wedge\beta) = (\mathrm{div}(\phi)\cdot\alpha)\wedge\beta.\]
		\item[(2)] (Commutativity law) For any$ \psi\in PS^{0,0}(\widetilde{\Omega})$, we have
		\[\mathrm{div}(\psi)\cdot(\mathrm{div}(\phi)\cdot\alpha) = \mathrm{div}(\phi)\cdot(\mathrm{div}(\psi)\cdot\alpha).\]
	\end{enumerate}
	\label{basic property of corner locus}
\end{corollary}

As showed in \cite{mihatsch2021on}, we have another way to define $\wedge$-product. Recall, two polyhedra $\sigma_1,\sigma_2$ intersect transversally if 
\[\mathrm{relint}(\sigma_1)\cap \mathrm{relint}(\sigma_2)\not=\emptyset\ \ \text{ and } \ \ N_{\sigma_1,\R}+N_{\sigma_2,\R} = N_\R.\]
We say two polyhedral complex $\mathcal{C}_1, \mathcal{C}_2$ of pure dimension $p, q$ respectively  intersect transversally if $|\mathcal{C}_1|\cap|\mathcal{C}_2|\not=\emptyset$, and for any $\sigma_1\in(\mathcal{C}_1)_p$, $\sigma_2\in(\mathcal{C}_2)_q$, we have $\sigma_1\cap \sigma_2=\emptyset$ or $\sigma_1, \sigma_2$ intersect transversally.

\begin{lemma}
	Let $\alpha=\sum\limits_{\sigma\in\mathcal{C}_1^{l_1}} \alpha_\sigma[\sigma], \beta=\sum\limits_{\sigma\in\mathcal{C}_2^{l_2}} \beta_\sigma[\sigma]\in B(\widetilde{\Omega})$ such that $\mathcal{C}_1^{\geq l_1}, \mathcal{C}_2^{\geq l_2}$ intersect transversally. Then
	\[\alpha\wedge\beta = \sum\limits_{\substack{(\sigma_1,\sigma_2)\in \mathcal{C}_1^{l_1}\times \mathcal{C}_2^{l_2}\\ \sigma_1\cap\sigma_2\not=\emptyset}}(\alpha_{\sigma_1}\wedge\beta_{\sigma_2})[\sigma_1\cap \sigma_2].\]
	\label{wedge product of transersally intersection}
\end{lemma}
\begin{proof}
	This is \cite[Lemma~4.15]{mihatsch2021on}.
\end{proof}

For $T\in D(N_\R)$ and $v\in N_\R$, we write $\lambda_v: N_\R\rightarrow N_\R,  \ \ x\mapsto x+v$  and denote by $v+T=\lambda_{v,*}T$ the $v$-translated supercurrent.

\begin{proposition}
	Let $\alpha, \beta\in B(\widetilde{\Omega})$ and $v\in N_\R$. Assume that there is $\widetilde{\beta}\in B(N_\R)$ such that $\widetilde{\beta}|_{\widetilde{\Omega}}=\beta$. Then we have a weak convergence
	\[\alpha\wedge \beta = \lim\limits_{\varepsilon\to 0}\alpha\wedge(\varepsilon v+\widetilde{\beta})|_{\widetilde{\Omega}}.\]
	\label{limit construction of wedge product}
\end{proposition}
\begin{proof}
	Since $\Delta = d'd''\phi_1\wedge\cdots \wedge d'd''\phi_r$ with $\phi_i=\max\{x_i, y_i\}$, and supercurrents are continuous with respect to the so-call $C^\infty$-topology, by equality~(\ref{wedge product1}), it is sufficient to show that $\alpha\times \beta = \lim\limits_{\varepsilon\to 0}\alpha\times(\varepsilon v+\widetilde{\beta})|_{\widetilde{\Omega}}$, i.e. 
	\[\alpha(\eta)\beta(\eta') = \alpha(\eta)\lim\limits_{\varepsilon\to 0}\widetilde{\beta}(\lambda_{\varepsilon v}^*\eta')\]
	for any $\eta,\eta'\in A_c(\widetilde{\Omega})$, where $\lambda_{\varepsilon v}: N_\R\rightarrow N_\R, \ \ x\mapsto x+\varepsilon v$. This is obviously, since $\Supp(\lambda_{\varepsilon v}^*\eta')\subset \widetilde{\Omega}$ for $\varepsilon$ small enough and $\eta'=\lim\limits_{\varepsilon\to 0}\lambda_{\varepsilon v}^*\eta'$ with respect to the $C^\infty$-topology.
\end{proof}

\begin{corollary}
	Let $\alpha=\sum\limits_{\sigma\in \mathcal{C}_1^{l_1}}\alpha_\sigma[\sigma], \beta=\sum\limits_{\sigma\in\mathcal{C}_2^{l_2}} \beta_\sigma[\sigma]\in B(\widetilde{\Omega})$ and $v\in N_\R$ a generic vector for $\mathcal{C}_1,$ $\mathcal{C}_2$. Then for any sufficiently small $\varepsilon>0$, we have
	\[\alpha\wedge\beta = \sum\limits_{\substack{(\sigma_1,\sigma_2)\in \mathcal{C}_1^{l_1}\times \mathcal{C}_2^{l_2}\\ \sigma_1\cap(\varepsilon v+\sigma_2)\not=\emptyset}}(\alpha_{\sigma_1}\wedge\beta_{\sigma_2})[\sigma_1\cap \sigma_2].\]
	\label{another definition of wedge product}
\end{corollary}
\begin{proof}
	Since the presheaf of polyhedral supercurrents is separated, it is sufficient to show the result locally. For any $x\in \widetilde{\Omega}$ and a small neighborhood $U_x$, we can find $\widetilde{\beta}\in B(N_\R)$ such that $\widetilde{\beta}|_{U_x} = \beta|_{U_x}$. This is from the fact $B$ is a $C^\infty$-module. So we can assume that there is $\widetilde{\beta}=\sum\limits_{\sigma\in\mathcal{C}_2^{l_2}} \widetilde{\beta}_\sigma[\sigma]\in B(N_\R)$ such that $\widetilde{\beta}|_{\widetilde{\Omega}}=\beta$. Then \cref{wedge product of transersally intersection} and \cref{limit construction of wedge product} imply that
	\[\alpha\wedge(\varepsilon v+ \widetilde{\beta})|_{\widetilde{\Omega}}=\sum\limits_{\substack{(\sigma_1,\sigma_2)\in \mathcal{C}_1^{l_1}\times \mathcal{C}_2^{l_2}\\ \sigma_1\cap(\varepsilon v+\sigma_2)\not=\emptyset}}(\alpha_{\sigma_1}\wedge(\varepsilon v+\widetilde{\beta}_{\sigma_2})|_{\widetilde{\Omega}})[\sigma_1\cap (\varepsilon v+\sigma_2)]\] 
	is weakly convergent to $\alpha\wedge\beta$. Moreover, we have the weak convergence
	\[(\alpha_{\sigma_1}\wedge(\varepsilon v+\widetilde{\beta}_{\sigma_2})|_{\widetilde{\Omega}})[\sigma_1\cap (\varepsilon v+\sigma_2)]\to (\alpha_{\sigma_1}\wedge{\beta}_{\sigma_2})[\sigma_1\cap (\sigma_2)], \ \ \varepsilon \to 0.\]
	Hence the corollary holds.
\end{proof}

\subsection{Integration of $\delta$-forms over tropical cycles}

\label{integration of delta-forms on open subsets}

We have defined integration for polyhedral currents in \cref{definition of integration for polyhedral supercurrents}. Here, we will study some basic properties for integration of $\delta$-forms.

\begin{definition}
	Let ${\alpha}\in B(\widetilde{\Omega})$, and $C\in P_n(\widetilde{\Omega})$ a polyhedral complex with constant weight. We say $C$ is {\bf balanced for ${\alpha}$} if there is a polyhedral complex $\mathcal{C}$ adapted to $C$  satisfying the following property: if $\tau$ is a facet of some maximal polyhedron in $\mathcal{C}$ with $\Supp_{\widetilde{\Omega}}({\alpha})\cap\tau\not=\emptyset$, then the balance condition holds at $\tau$. Here $\Supp_{\widetilde{\Omega}}({\alpha})$ is the support of ${\alpha}$ in $\widetilde{\Omega}$.
	\label{def:balancedfordeltaform}
\end{definition}
\begin{remark}
	\begin{enumerate}
		\item [(1)] Although $C$ may not be balanced on $\widetilde{\Omega}$, if it is balanced for ${\alpha}$, we can take an open subset $U \subset \widetilde{\Omega}$ such that $\Supp_{\widetilde{\Omega}}(\alpha)\subset U$ and $C$ is balanced on $U$. In this case, we can take  ${\alpha}|_U\wedge C\in B(U)$. We will omit $U$ since the property that it is zero is independent of the choice of $U$. 
		
		If, moreover, $\alpha\in B^{n,n}(\widetilde{\Omega})$ and  $\Supp_{\widetilde{\Omega}}({\alpha})\cap |C|$ is compact, then we set
		\[\int_{\widetilde{\Omega}}\alpha\wedge C = \int_U\alpha|_U\wedge C,\]
		this is independent of the choice of $U$.
		\label{rk:wedgezero}
	\end{enumerate}
\end{remark}

\

We have seen that the Stokes' formula, i.e. \cref{stokes' formula for polyhedral supercurrents}, hold for polyhedral currents when we take polyhedral derivatives. However, the formula doesn't hold when we consider derivative on $\delta$-forms, since
\[\int_Pd'\alpha = \int_Pd_P'\alpha+\int_P\partial'\alpha= \int_{\partial P}\alpha+\int_{P}\partial'\alpha.\] 
By definition of derivative, we have the following observation.

\begin{proposition}
	Let $C\in P_n(N_\R)$ be a polyhedral complex with constant weight, and $\alpha\in B_c^{n-1,n}(\widetilde{\Omega})$ (resp. $\alpha\in B_c^{n,n-1}(\widetilde{\Omega})$). Assume that there is a polyhedral complex $\mathcal{C}$ adapted to $C$ satisfying the following property: if $\tau\in \mathcal{C}_{n-1}$ with $\Supp_{\widetilde{\Omega}}(\alpha)\cap\tau\not=\emptyset$, then the balance condition holds at $\tau$. Then we have 
	\[\int_{\widetilde{\Omega}}(d'\alpha)\wedge C = \int_{\widetilde{\Omega}}d_P'\alpha\wedge C= \int_{\widetilde{\Omega}}\partial'\alpha\wedge C=0,\]
	\[\text{(resp.$\int_{\widetilde{\Omega}}d''\alpha\wedge C = \int_{\widetilde{\Omega}}d_P''\alpha\wedge C = \int_{\widetilde{\Omega}}\partial''\alpha\wedge C=0$)}.\]
	\label{minor stokes' formula for d'd''}
\end{proposition}
\begin{proof}
	We take an open neighborhood $U\subset\widetilde{\Omega}$ of $\Supp(\alpha)$ such that $\Supp(\alpha)\cap\tau=\emptyset$ implies that $\tau\cap U=\emptyset$. So $C\in B^{0,0}(U)$ by our assumption. 
	Let $f\in A^{0,0}(U)$ such that $f|_{\Supp(\alpha)} \equiv 1$. Notice that $d'C=0$ on $U$, then
	\[\int_{C}d'\alpha=\int_{U}C\wedge d'\alpha = \int_{U} d'(C\wedge\alpha)= (d'(C\wedge\alpha))(f) = (C\wedge\alpha)(d'f) = (C\wedge\alpha)(0)=0.\]
	For $\int_{C}d_P'\alpha$, we take a polyhedral set $P\subset U$ such that $\Supp(\alpha)\subset \mathrm{Int}(P)$. Then by Stokes' formula, i.e. \cref{stokes' formula for polyhedral supercurrents}, we have 
	\[\int_{C}d_P'\alpha = \int_{U}C\wedge d_P'\alpha=\int_Pd_P'(C\wedge\alpha)= \int_{\partial P}C\wedge\alpha = 0.\]
	These two imply that $\int_{C}\partial'\alpha=0$.
\end{proof}



\

We have Green's formula for $\delta$-forms with respect to $d_P, d_P'$ by \cref{stokes' formula for polyhedral supercurrents}, the proof is similar with \cref{Green's formula for superforms}. We need the following lemma.

\begin{lemma}
	For any $\alpha\in B^{p,q}(\widetilde{\Omega}), \beta\in B^{p',q'}(\widetilde{\Omega})$, we have
	\[J(\alpha\wedge\beta) = J\alpha\wedge J\beta.\]
\end{lemma}
\begin{proof}
	Notice that $J(\sum\limits_{\sigma\in \mathcal{C}^l}\alpha_\sigma[\sigma]) = (-1)^l\sum\limits_{\sigma\in \mathcal{C}^l}J(\alpha_\sigma)[\sigma]$, then the result from Lemma\ref{basic relation of d' d'' and J} and Corollary~\ref{another definition of wedge product}.
\end{proof}

\begin{proposition}[Green's formula for $\delta$-forms]
	Let $C\in F_r(N_\R)$ with $|C|\subset \widetilde{\Omega}$ an integral $\R$-affine polyhedral set in $N_\R$. Then for any symmetric $(p,p)$-$\delta$-form $\alpha$, symmetric $(q,q)$-$\delta$-form $\beta$ around $|C|$ with $p+q=r-1$ and $\Supp(\alpha)\cap\Supp(\beta)$ compact,
	\[\int_{C}\alpha\wedge d_P'd_P''\beta-d_P'd_P''\alpha\wedge \beta = \int_{\partial C}\alpha\wedge d_P''\beta-d_P''\alpha\wedge\beta,\]
	\label{Green' formula for delta-forms}
\end{proposition}

\section{Push-forwards and pull-backs of $\delta$-forms}

\label{pushforwardandpullback}

In this section, we will study the functorial properties of $\delta$-forms. We have defined push-forwards of $\delta$-forms as supercurrent in \cref{direct image of current}, in fact, we have another definition of push-forwards of $\delta$-forms as tropical cycles in tropical geometry. These two definitions are slightly different. We will also define pull-backs of $\delta$-forms and prove the projection formulas.

\subsection{Push-forwards as supercurrents}


\begin{lemma}
	Let $F: N'_\R\rightarrow N_\R$ an integral $\R$-affine map, and $\sigma'\subset N'_\R$ an integral $\R$-affine polyhedron.  We set $k=\dim{\sigma'}-\dim{F(\sigma')}$. For any $\alpha'\in A^{p,q}_{\sigma',c}(\sigma')$.
	Then we can construct $\widetilde{F}_*({\alpha}')\in A^{p-k,q-k}_{F(\sigma'),c}(F(\sigma'))$ as follows: for any $x\in N_{F(\sigma'),\R}$, and $v_{1}',\cdots, v_{p-k}', v_{1}'',\cdots, v''_{q-k}\in N_{F(\sigma'),\R}$,
	\[\widetilde{F}_*({\alpha}')(x)(v_1',\cdots,v_{p-k}', v_{1}'',\cdots v_{q-k}''):= \int_{\sigma'\cap F^{-1}(x)}\beta,\]
	where 
	\[\beta := (-1)^{k^2+kp}\langle{\alpha}';((u'_1,\cdots, u'_{p-k})_{\{k+1,\cdots, p\}}, (u''_1,\cdots, u''_{q-k})_{\{k+1,\cdots, q\}})\rangle|_{\sigma'\cap F^{-1}(x)}\]
	and $u_{1}',\cdots, u_{p-k}', u_{1}'',\cdots, u''_{q-k}$ are preimages of $v_{1}',\cdots, v_{p-k}', v_{1}'',\cdots, v''_{q-k}$ in $N_{\sigma',\R'}$ via $dF$.
	
	Furthermore, $F_*$ is $\R$-linear and for any $\eta\in A^{\dim\sigma'-p,\dim\sigma'-q}_{F(\sigma')}(F(\sigma'))$, we have 
	\[[N_{F(\sigma')}:dF(N_{\sigma'}')]\int_{F(\sigma')}\widetilde{F}_*\alpha'\wedge\eta= \int_{\sigma'} \alpha'\wedge {F}^*\eta.\]
	\label{fiberwise integration between polyhedra}
\end{lemma}
\begin{proof}
	We can assume that $N'=N'_{\sigma'}$ and $N=N_{F(\sigma')}$.
	By definition, we know that $\tau':=\sigma'\cap F^{-1}(x)$ is an integral $\R$-affine polyhedron. We first show that $\beta \in A^{k,k}_{\tau',c}(\tau')$ is well-defined. This is from the fact that $N'_{F^{-1}(x),\R}$ as a vector space is independent of the choice of $x$, and it is exactly $\Ker(dF)$. Hence the integral is well-defined. The value is smooth since $\alpha'$ is smooth. Moreover $\Supp(F_*(\alpha')) \subset F(\Supp(\alpha'))$, so $F_*(\alpha')$ has compact support.
	
	The last statement is the Fubini's theorem for smooth manifolds, see \cite[2.15]{demailly2012complex}. We go through the detail to verity the sign.  
	We have that $N'=M\times dF(N')$ and $F: N_\R'\rightarrow N_\R$ is the projection. Notice that the integration for $dF(N')$ is the integration for $N$ multiplying $[N:dF(N')]$, see \cref{integral via bijection}, so we can assume that $N'=M\times N$. Moreover, assume that $\sigma'=N'_\R$ since $A^{p,q}_{\sigma',c}(\sigma') \subset A^{p,q}_{c}(\mathrm{relint}(\sigma'))\subset A^{p,q}_{c}(N_\R')$. Let $x_1,\cdots x_k$ (resp. $x_{k+1}, \cdots, x_{k+r}$) be the coordinate functions of $M_\R$ (resp. $N_\R$). Let $\alpha' = \sum\limits_{I,J}\alpha'_{IJ}d'x_Id''x_J$. Set $T=\{1,\cdots, k\}$. By definition 
	\begin{align*}
		\widetilde{F}_*\alpha'&=(-1)^{k^2+kp}\sum\limits_{\widetilde{I}, \widetilde{J}}(\int_{M_\R}\alpha_{IJ}(x_1,\cdots,x_{k+r})d'x_T\wedge d''x_T)dx_{\widetilde{I}}\wedge d'x_{\widetilde{J}}\\
		&=(-1)^{k^2+kp+k(k-1)/2}\sum\limits_{\widetilde{I}, \widetilde{J}}(\int_{M_\R}\alpha_{IJ}'(x_1,\cdots,x_{k+r})dx_T)dx_{\widetilde{I}}\wedge d'x_{\widetilde{J}}
	\end{align*}
	where $\widetilde{I}, \widetilde{J}$ run over subsets of $\{r+1,\cdots, r+k\}$ with $|\widetilde{I}|=p-k, |\widetilde{J}|=q-k$, and $I=\widetilde{I}\cup T$, $J=\widetilde{J}\cup T$. Assume $\eta=\sum\limits_{\overline{I},\overline{J}}\eta_{\overline{I}\overline{J}}d'x_{\overline{I}}d''x_{\overline{J}}\in A^{k+r-p,k+r-q}_{F(\sigma')}(F(\sigma'))$, then
	\begin{align*}
		\int_{N_\R'} \alpha'\wedge F^*\eta &= \sum\limits_{\widetilde{I},\widetilde{J}}\int_{N_\R'} \alpha'_{IJ}\eta_{\widetilde{I}^c\widetilde{J}^c}d'x_I\wedge d''x_{J}\wedge d'x_{\widetilde{I}^c}\wedge d''x_{\widetilde{J}^c}\\
		&= a\cdot\sum\limits_{\widetilde{I},\widetilde{J}}c_{\widetilde{I}}c_{\widetilde{J}}\int_{N_\R'} \alpha'_{IJ}\eta_{\widetilde{I}^c\widetilde{J}^c}dx_1\wedge\cdots \wedge dx_{r+k}\\
	\end{align*}
	\begin{align*}
		\int_{N_\R}\widetilde{F}_*\alpha'\wedge\eta & = (-1)^{k^2+kp+k(k-1)/2}\int_{N_\R}\sum\limits_{\widetilde{I}, \widetilde{J}}(\int\alpha_{IJ}'dx_T)\eta_{\widetilde{I}^c\widetilde{J}^c}d'x_{\widetilde{I}}\wedge d''x_{\widetilde{J}}\wedge d'x_{\widetilde{I}^c}\wedge d''x_{\widetilde{J}^c}\\ 
		& =b\cdot\sum\limits_{\widetilde{I}, \widetilde{J}}c_{\widetilde{I}}c_{\widetilde{J}}\int_{N_\R}(\int_{M_\R}\alpha_{IJ}'\eta_{\widetilde{I}^c\widetilde{J}^c}dx_T)dx_{r+1}\wedge\cdots \wedge dx_{r+k}
	\end{align*}
	where $\widetilde{I}^c=\{r+1,\cdots, r+k\}\setminus\widetilde{I}$, $c_{\widetilde{I}}$ is the sign of the permutation \[(r+1,\cdots, r+k)\rightarrow (\widetilde{I}, \widetilde{I}^c),\] it is similar for $\widetilde{J}^c, c_{\widetilde{J}}$, and \[a=(-1)^{q(k+r-p)+(k+r)(k+r-1)/2},\ \  b=(-1)^{k^2+kp+k(k-1)/2}\cdot(-1)^{(q-k)(k+r-p)+r(r-1)/2}.\] By Fubini theorem, it remains to show that $a=b$, this is an easy calculation.
\end{proof}
\begin{remark}
	\begin{enumerate}
		\item[(1)] If $\dim F(\sigma') = \dim \sigma'$, we have $\widetilde{F}_* = ((F|_{\sigma})^{-1})^*$, see also \cref{integral via bijection}.
	\end{enumerate}
\end{remark}

\begin{corollary}
	Let $F: N'_\R\rightarrow N_\R$ be an integral $\R$-affine map. Let $\widetilde{\Omega}\subset N_\R$, $\widetilde{\Omega}'\subset N'_\R$ be open subsets with $F(\widetilde{\Omega}')\subset \widetilde{\Omega}$. Then for any polyhedral supercurrent $\alpha'[\sigma']\in P_c^{p,q,l}(\widetilde{\Omega}')$ we have
	\[F_*(\alpha'[\sigma']) = [N_{F(\sigma')}:dF(N_{\sigma'}')]\widetilde{F}_*\alpha'[F(\sigma')] \in P_c^{p-k,q-k,l+\rk N-\rk N'+k}(\widetilde{\Omega}),\]
	where $F_*$ on the left-hand side (resp. right-hand side) is defined in \cref{direct image of current} (resp. \cref{fiberwise integration between polyhedra}) and $k=\dim\sigma'-\dim F(\sigma')$. In particular, $F_*: P_c^{p,q,l}(\widetilde{\Omega}') \rightarrow P_c^{p-k,q-k,l+r-r'-k}(\widetilde{\Omega})$ induces
	\[F_*: B_c^{p,q,l}(\widetilde{\Omega}') \rightarrow B_c^{p-k,q-k,l+r-r'+k}(\widetilde{\Omega})\]
	and commutes with $d', d'', d'_P, d''_P, \partial'$ and $\partial''$.
	\label{push-forward of polyhedral supercurrents} 
\end{corollary}
\begin{proof}
	Notice that $\alpha'\in A_{\sigma',c}^{p,q}(\widetilde{\Omega}'\cap \sigma')\subset A^{p,q}_{\sigma',c}(\sigma')$,  by \cref{fiberwise integration between polyhedra}, $\widetilde{F}_*\alpha'\in A_{F(\sigma'),c}^{p-k,q-k}(F(\sigma'))$ with  \[\Supp(\widetilde{F}_*\alpha') \subset F(\Supp(\alpha'))\subset F(\widetilde{\Omega}'\cap \mathrm{relint}(\sigma))\subset \widetilde{\Omega}\cap\mathrm{relint}(F(\sigma')),\] so $\widetilde{F}_*\alpha'\in A^{p,q}_{F(\sigma'),c}(\widetilde{\Omega}\cap F(\sigma'))$. Moreover, for any $\eta\in A^{r'-l-p, r'-l-q}_c(\widetilde{\Omega})\subset A^{r'-l-p,r'-l-q}_c(N_\R)$, we have 
	\[[N_{F(\sigma')}:dF(N_{\sigma'}')]\int_{F(\sigma')}\widetilde{F}_*\alpha'\wedge\eta= \int_{\sigma'} \alpha'\wedge F^*\eta,\]
	i.e.\[[N_{F(\sigma')}:dF(N_{\sigma'}')]\widetilde{F}_*\alpha'[F(\sigma')]=F_*(\alpha'[\sigma']).\]
	
	Since $d'F_*=F_*d', d''F_*=F_*d''$, we can see that $F_*$ maps $\delta$-forms to $\delta$-forms. Moreover, It is not hard to see that $F_*d'_P = d'_PF_*$, $F_*d''_P = d''_PF_*$.
\end{proof}

\subsection{Push-forward from tropical geometry}

From corollary above, we see that, the push-forward $F_*$ is slightly different from the classical definition in tropical geometry: in tropical geometry, $F_*([\sigma]) =0$ if $\dim F(\sigma)< \dim\sigma$. So we have another definition of push-forward.

\begin{proposition}
	Let $F: N_\R'\rightarrow N_\R$ be an integral $\R$-affine map. Then the following hold.
	\begin{enumerate}
		\item [(1)] $\widehat{F}_*: P_n^{p,q}(N_\R')\rightarrow P_n^{p,q}(N_\R)$ is a well-defined linear map (independent of the choice of representative $(\mathcal{C},\alpha_\sigma)$) satisfying the following properties: for any $\alpha_\sigma[\sigma]$ with $\alpha_\sigma\in A^{p,q}_\sigma(\sigma)$,
		\begin{itemize}
			\item if $\dim F(\sigma)<n$, then $\widehat{F}_*(\alpha_\sigma[\sigma])=0$; 
			\item if $\dim F(\sigma)=n$, then $\widehat{F}_*(\alpha_\sigma[\sigma])= [N_{F(\sigma)}: (dF)(N_\sigma')]((F|_{\sigma})^{-1})^*(\alpha_\sigma)[F(\sigma)]$ (notice that $(dF)(N_\sigma)\subset N_{F(\sigma)}'$ is of rank $n$), where $F|_{\sigma}: N_{\sigma,\R}'\rightarrow N_{\sigma,\R}$ is bijective.
		\end{itemize}
		Obviously, $\widehat{F}_*$ is the unique linear map with these properties. In particular, for any $\alpha\in P_n^{p,q}(N_\R)$, one has $\Supp(\widehat{F}_*(\alpha))\subset F(\Supp(\alpha))$.
		
		\item[(2)] If $\alpha\in B(N_\R')$, then $\widehat{F}_*(\alpha)\in B(N_\R)$. Hence, we have a natural homomorphism
		\[\widehat{F}_*: B^{p,q}_n(N_\R')\rightarrow B^{p,q}_n(N_\R),\]
		which commutes with $d', d'', d_P', d_P'', \partial'$ and $\partial''$. 
	\end{enumerate}
\end{proposition}
\begin{proof}
	
	(1) This is obvious from \cref{remark:polyhedralcomplexwithsmoothforms}.
	
	(2) On $B_n^{p,q}(N_\R')$, locally we have $\widehat{F}_* = \pr\circ F_*$, where $\pr: B(N_\R)\rightarrow B_n^{p,q}(N_\R)$. Then result is from \cref{push-forward of polyhedral supercurrents}.
\end{proof}
\begin{remark}
	\begin{enumerate}
		\item [(1)] The definition of $\widehat{F}_*$ extends the classical push-forward on tropical cycles to $\delta$-forms. If $\alpha$ is a polyhedral supercurrent such that  there is a polyhedral decomposition with $F|_{\sigma}$ injective, then $F_*\alpha = \widehat{F}_*\alpha$.
	\end{enumerate}
\end{remark}

\subsection{Pull-backs of $\delta$-forms}

\begin{definition}
	Let $F: (N',\widetilde{\Omega}')\rightarrow (N,\widetilde{\Omega})$ be a morphism in $\mathcal{AT}op_{\Z,\R}$.  Then we define a bilinear map \begin{align*}
		(\cdot)\wedge F^*(\cdot): B(\widetilde{\Omega}')\times B(\widetilde{\Omega})&\rightarrow B(\widetilde{\Omega}'),\\
		(\alpha',\alpha) &\mapsto p_{1,*}([\Delta_F]\wedge(\alpha'\times\alpha)),
	\end{align*} 
	where $\Delta_F\subset N_\R'\times N_\R$ is the graph of $F$ and $p_1: N_\R'\times N_\R\rightarrow N_\R'$ is the projection.
	In particular, we have a {\bf pull-back} of $\delta$-forms
	\begin{align*}
		F^*: B(\widetilde{\Omega})&\rightarrow B(\widetilde{\Omega}'),\\
		\alpha&\mapsto [N_\R']\wedge F^*\alpha.
	\end{align*}
	\label{pull-back of delta-forms}
\end{definition}
\begin{remark}
	\begin{enumerate}
		\item [(1)]  We fix integral coordinates  $(x_1',\cdots, x_{r'}')$, $(x_1,\cdots, x_{r})$ of $N_\R'$ and $N_\R$ respectively. By \cref{tropical poincare-lelong formula}, we have 
		\[[\Delta_F]=d'd''\varphi_1\wedge\cdots \wedge d'd''\varphi_r,\]
		where $\varphi_i=\max\{f_i(x_1',\cdots, x_{r'}'), x_i\}$ and $F=(f_1,\cdots, f_r): N_\R'\rightarrow N_\R$.
		\item[(2)] From \cref{product of currents}, it is not hard to see that $F^*$ commutes with $d', d'', d'_P, d_P'', \partial'$ and $\partial''$.
	\end{enumerate}
\end{remark}

\begin{lemma}
	Keep the notion in \cref{pull-back of delta-forms}. 
	\begin{enumerate}
		\item[(1)] If $\alpha=\{\alpha_{\sigma}\}_{\sigma\in\mathcal{C}}\in PS^{p,q}(\widetilde{\Omega})$, then $F^*\alpha = \{F^*\alpha_\sigma\}_{\sigma\in\mathcal{C}}$; if $\alpha\in B^{0,0,l}(\widetilde{\Omega})$, then $F^*\alpha$ is the pull-back of tropical cycle $\alpha$.
		\item [(2)] If $\alpha = \mathrm{div}(\phi_1)\cdots \mathrm{div}(\phi_n)\cdot [N_{\R}]$ for some $\phi_1,\cdots, \phi_n\in PS^{0,0}(\widetilde{\Omega})$, then $F^*\alpha = \mathrm{div}(F^*\phi_1)\cdots \mathrm{div}(F^*\phi_n)\cdot [N_{\R}']$.
		\item[(3)]  For any $\phi\in PS^{0,0}(\widetilde{\Omega})$ and $\alpha'\in B_c(\widetilde{\Omega}')$, we have
		\[F_*(\mathrm{div}(F^*\phi)\cdot \alpha') = \mathrm{div}(\phi)\cdot F_*\alpha'.\]
	\end{enumerate}
	\label{basic properties of pull-back1}
\end{lemma}
\begin{proof}
	(1) If $\alpha$ is a tropical cycle, then $F^*\alpha$ is the pull-back of a tropical cycle by definition. If $\alpha=\{\alpha_{\sigma}\}_{\sigma\in\mathcal{C}}\in PS^{p,q}(\widetilde{\Omega})$, then for any $\gamma\in A^{r-p,r-q}_c(\widetilde{\Omega})$, we have
	\begin{align*}
		(F^*\alpha)(\gamma)&= (\Delta_F\wedge([N_\R]\times \alpha))(p_1^*\gamma)\\
		&= \sum\limits_{\sigma\in \mathcal{C}^0}(p_2^*\alpha_\sigma[\Delta_F\cap(N_\R'\times\sigma)])(p_1^*\gamma)\\
		&=\sum\limits_{\sigma\in \mathcal{C}^0}\int_{\Delta_F\cap(N_\R'\times\sigma)}(p_2i)^*\alpha_\sigma\wedge (p_1i)^*\gamma\\
		&=\sum\limits_{\sigma\in \mathcal{C}^0}\int_{F^{-1}(\sigma)}(\widetilde{p_{1}i})_*(p_2i)^*\alpha_\sigma\wedge \gamma\\
		&=\sum\limits_{\sigma\in \mathcal{C}^0}\int_{F^{-1}(\sigma)}F^*\alpha_\sigma\wedge \gamma,
	\end{align*}
	where $i: \Delta_F\rightarrow N_\R'\times N_\R$ is the natural embedding.
	The second equality is given by the fact that $\Delta_F$ and $N_\R'\times\sigma$ intersect transversally and \cref{another definition of wedge product}. The fourth equality follows from \cref{fiberwise integration between polyhedra}, and fifth follows the commutative diagram
	\[\xymatrix{N'_\R \ar@<.5ex>[r]^{(\id,F)}  \ar[dr]_F & \Delta_F\ar[d]^{p_2i} \ar@<.5ex>[l]^{p_1i}\\
		&N_\R}\]
	
	(2) We follow the proof of \cite[Theorem~3.3~(c)]{allermann2012tropical}
	\begin{align*}
		F^*\alpha &= p_{1,*}([\Delta_F]\wedge([N_\R]\times(\mathrm{div}(\phi_1)\cdots \mathrm{div}(\phi_n)\cdot [N_{\R}]))\\
		&= p_{1,*}(\mathrm{div}(p_{2}^*\phi_1)\cdots \mathrm{div}(p_2^*\phi_n)\cdot ([\Delta_F]\wedge [N_\R'\times N_{\R}]))\\
		&= p_{1,*}(\mathrm{div}(p_{2}^*\phi_1)\cdots \mathrm{div}(p_2^*\phi_n)\cdot [\Delta_F])\\
		&= \mathrm{div}(F^*\phi_1)\cdots \mathrm{div}(F^*\phi_n)\cdot [N_{\R}'].
	\end{align*}
	The final equality comes from the commutative diagram above.
	
	(3) For any $\alpha\in PS^{p,q}(\widetilde{\Omega})$ and $\alpha'\in B_c(\widetilde{\Omega}')$, we have
	\begin{align}
		F_*(F^*\alpha\wedge\alpha') = \alpha\wedge F_*\alpha'.
		\label{projection formula for piecewise smooth}
	\end{align}
	Indeed, after refinement, we can write $\alpha=\{\alpha_\sigma\}_{\sigma\in \mathcal{C}}, \alpha'=\sum\limits_{\sigma'\in \mathcal{C}'}\alpha'_{\sigma'}[\sigma']$ such that $F(\mathcal{C}')\subset\mathcal{C}$. For any $\gamma\in A_c(\widetilde{\Omega}')$, by \cref{wedge product for ps and delta-form} and (1), we have
	\begin{align*} 
		F_*(F^*\alpha\wedge\alpha')(\gamma)&= (F^*\alpha\wedge\alpha')(F^*\gamma)
		= \sum\limits_{\sigma'\in \mathcal{C}'}\int_{\sigma'}F^*\alpha_{\sigma}\wedge\alpha'_{\sigma'}\wedge F^*\gamma\\
		(\alpha\wedge F_*\alpha')(\gamma)&=\sum\limits_{\sigma\in\mathcal{C}}\sum\limits_{\substack{\sigma'\in \mathcal{C}'\\ F(\sigma')=\sigma}}\alpha_{\sigma}\wedge F_*(\alpha'_{\sigma'}[\sigma'])(\gamma)= \sum\limits_{\sigma\in\mathcal{C}}\sum\limits_{\substack{\sigma'\in \mathcal{C}'\\ F(\sigma')=\sigma}}\int_{\sigma'}F^*\alpha_\sigma\wedge\alpha'_{\sigma'}\wedge F^*\gamma. 
	\end{align*}
	Let $\phi\in PS^{0,0}(\widetilde{\Omega})$. Recall that $\mathrm{div}(\phi)\cdot \alpha =  \partial'(d_P''\phi\wedge\alpha)+d_P''\phi\wedge\partial'\alpha$. Since $F_*$ commutes with $\partial'$, and $F^*$ commutes with $d_P', d_P''$, we have
	\begin{align*}
		F_*(\mathrm{div}(F^*\phi)\cdot \alpha') &= F_*(\partial'(d_P''F^*\phi\wedge\alpha')+ d_P''(F^*\phi)\wedge \partial'\alpha')\\
		&= \partial'(F_*(F^*d_P''\phi\wedge\alpha'))+ F_*(F^*(d_P''\phi)\wedge \partial'\alpha')\\
		&= \partial'(d_P''\phi\wedge F_*\alpha')+ d_P''\phi\wedge \partial'F_*\alpha'\\
		&= \mathrm{div}(\phi)\cdot F_*\alpha'.
	\end{align*} 
	The third equality is from equality~(\ref{projection formula for piecewise smooth}).
\end{proof}
\begin{remark}
	\begin{enumerate}
		\item [(1)] By partition of unity, we know that the property (3) holds as soon as both sides make sense.
	\end{enumerate}
\end{remark}

\begin{lemma}
	Let $F: (N',\widetilde{\Omega}')\rightarrow (N, \widetilde{\Omega})$, $G: (N'',\widetilde{\Omega}'')\rightarrow (N',\widetilde{\Omega}')$ be morphisms in $\mathcal{AT}op_{\Z,\R}$. Then for any $\alpha\in B(\widetilde{\Omega})$, $\alpha'\in B(\widetilde{\Omega}')$, $\alpha''\in B(\widetilde{\Omega}'')$, we have $\alpha''\wedge G^*(\alpha'\wedge F^*\alpha) = (\alpha''\wedge G^*\alpha')\wedge (FG)^*\alpha$.
	\label{associativity of pull-back intersection}
\end{lemma}
\begin{proof}
	We give some properties we will use during the proof.
	\begin{itemize}
		\item For any $(\alpha',\alpha)\in B(\widetilde{\Omega}')\times B(\widetilde{\Omega})$, and any reasonable maps $F_1, F_2$, we have \[(F_1\times F_2)_*(\alpha'\times \alpha) =F_{1,*}\alpha'\times F_{2,*}\alpha.\]
		\item For any $\alpha'\in B(\widetilde{\Omega}')$ and $\alpha_1,\alpha_2\in B(\widetilde{\Omega})$, we have \[\alpha'\times (\alpha_1\wedge \alpha_2)= ([N_\R']\times \alpha_1)\wedge (\alpha'\times \alpha_2).\]
		\item Let $\widetilde{\Delta} = \mathrm{div}(\phi_1)\cdots\mathrm{div}(\phi_n)\cdot [N_\R]$ for some $\phi_1,\cdots, \phi_n\in PS^{0,0}(\widetilde{\Omega})$ and $\alpha'\in B_c(\widetilde{\Omega}')$, we have 
		\[F_*(F^*\widetilde{\Delta}\wedge \alpha') = \widetilde{\Delta}\wedge F_*\alpha'\]
	\end{itemize}
	The first one is trivial, the second one is deduced from \cref{another definition of wedge product}. For the third, by \cref{basic properties of pull-back1}~(2)~(3), we have
	\begin{align*}
		F_*(F^*\widetilde{\Delta}\wedge \alpha')& =F_*(\mathrm{div}(F^*\phi_1)\cdots\mathrm{div}(F^*\phi_n)\cdot ([N_\R]\wedge \alpha'))\\&=\mathrm{div}(\phi_1)\cdots\mathrm{div}(\phi_n)\cdot F_*([N_\R]\wedge \alpha')\\&= \widetilde{\Delta}\wedge F_*\alpha'
	\end{align*}

	Let $\Delta_F\subset N_\R'\times N_\R$ and $\Delta_G\subset N_\R''\times N_\R'$ be the graphs of $F$ and $G$, respectively. We consider the following subsets of $N''_\R\times N_\R'\times N_\R$:
	\[\Delta':=\{(w'', G(w''), F(G(w'')))\mid w''\in N_\R''\},\]
	\[\Delta_{12} := \{(w'',G(w''), w)\mid w\in N_\R, w''\in N_\R''\},\]
	\[\Delta_{23}: = \{(w'', w', F(w'))\mid w'\in N_\R, w''\in N_\R''\},\]
	\[\Delta_{13}:=\{(w'', w', F(G(w'')))\mid w'\in N_\R, w''\in N_\R''\}.\]
	We also denote \[p_1: N_\R'\times N_\R\rightarrow N_\R',\] \[p_1': N_\R''\times N_\R'\rightarrow N_\R'',\] \[p_1'': N_\R''\times N_\R\rightarrow N_\R'',\]
	\[q_{1}: N''_\R\times N'_\R\times N_\R\rightarrow N_\R'',\] 
	\[q_{12}: N''_\R\times N'_\R\times N_\R\rightarrow N_\R''\times N_\R',\]
	\[q_{13}: N''_\R\times N'_\R\times N_\R\rightarrow N_\R''\times N_\R\]  
	the canonical projections. Then we have
	\begin{align*}
		\alpha''\wedge G^*(\alpha'\wedge F^*\alpha)& = p_{1,*}'(\Delta_G\wedge(\alpha''\times (\alpha'\wedge F^*\alpha)))\\
		& = p_{1,*}'(\Delta_G\wedge(\alpha''\times p_{1,*}(\Delta_F\wedge(\alpha'\times \alpha))))\\
		& = p_{1,*}'(\Delta_G\wedge q_{12,*}((N''_\R\times\Delta_F)\wedge(\alpha''\times\alpha'\times \alpha)))\\
		& = p_{1,*}'q_{12,*}(q_{12}^*\Delta_G\wedge (\Delta_{23}\wedge(\alpha''\times\alpha'\times \alpha)))\\
		& = q_{1,*}(\Delta_{12}\wedge \Delta_{23}\wedge (\alpha''\times\alpha'\times \alpha))\\
		& = q_{1,*}(\Delta'\wedge (\alpha''\times\alpha'\times \alpha)),
	\end{align*}
	\begin{align*}
		(\alpha''\wedge G^*\alpha')\wedge (FG)^*\alpha''& = p_{1,*}''(\Delta_{FG}\wedge ((\alpha''\wedge G^*\alpha')\times\alpha))\\
		& = p_{1,*}''(\Delta_{FG}\wedge (p_{1,*}'(\Delta_G\wedge (\alpha''\times\alpha'))\times\alpha))\\
		& = p_{1,*}''(\Delta_{FG}\wedge q_{13,*}((\Delta_G\times N_\R)\wedge (\alpha''\times\alpha'\times\alpha))\\
		& = p_{1,*}''q_{13,*}(q_{13}^*\Delta_{FG}\wedge (\Delta_{12}\wedge (\alpha''\times\alpha'\times\alpha))\\
		& = q_{1,*}(\Delta_{13}\wedge \Delta_{12}\wedge (\alpha''\times\alpha'\times\alpha))\\
		& = q_{1,*}(\Delta'\wedge (\alpha''\times\alpha'\times \alpha)),
	\end{align*}
	so we have the equality.
\end{proof}

\

Indeed, we have another way to define pull-backs based on the following proposition.

\begin{proposition}
	Let $F: (N',\widetilde{\Omega}')\rightarrow (N, \widetilde{\Omega})$ be a morphism in $\mathcal{AT}op_{\Z,\R}$, and $\alpha=\sum\limits_{\sigma\in\mathcal{C}^l} \alpha_\sigma[\sigma]\in B^{p,q,l}(\widetilde{\Omega})$.
	\begin{enumerate}
		\item [(1)] If $F=i: N_\R'\hookrightarrow N_\R$ is an embedding such that $i(N_\R')$ is an affine subspace of $N_\R$ and $N'=N\cap i(N_\R')$, then for any $\gamma\in A_c^{p,q}(\widetilde{\Omega})$, we have
		\[i^*\alpha(i^*\gamma) = ([i(N_\R')]\wedge\alpha)(\gamma),\]
		i.e.
		\[i^*\alpha= \sum\limits_{\substack{\sigma\in\mathcal{C}^l\\ N_\R'\cap (\varepsilon v+\sigma)\not=\emptyset}} \alpha_\sigma|_{\widetilde{\Omega}'}[i(N_\R')\cap \sigma],\]
		where $v\in N_\R$ is a generic vector, $\varepsilon>0$ is sufficiently small and $\alpha_\sigma|_{\widetilde{\Omega}'}\in A_{\sigma\cap i(N_\R')}(\widetilde{\Omega}'\cap \sigma\cap i(N_\R'))$ is the restriction of $\alpha_\sigma$.
		
		\item [(2)] If $F: N_\R'\rightarrow N_\R$ is surjective, then 
		\[F^*\alpha = \sum\limits_{\sigma\in \mathcal{C}^l}F^*\alpha_\sigma[F^{-1}(\sigma)].\]
		Moreover, for any $\gamma\in A_c^{\rk N'-p-l,\rk N'-q-l}(\widetilde{\Omega}')$, we have
		\[(F^*\alpha)(\gamma) = \sum\limits_{\sigma\in \mathcal{C}^l}([N_\sigma:dF(N_{F^{-1}(\sigma)})]\alpha_\sigma[\sigma])(F_*\gamma).\]  
	\end{enumerate}
\end{proposition}
\begin{proof}
	(1) For simplicity, we identify $N_\R'$ with $i(N_\R')$. Assume that $\widetilde{\Omega}'= N_\R'\cap \widetilde{\Omega}$ and $\alpha=\sum\limits_{\sigma\in \mathcal{C}^l}\alpha_\sigma[\sigma]$. Let $\Delta\subset N_\R\times N_\R$ be the diagonal and $\Delta' = \Delta\cap (N_\R'\times N_\R)$. For any $\gamma\in A_c^{p,q}(\widetilde{\Omega})$, we have
	\[(i^*\alpha)(i^*\gamma)=(\Delta'\wedge ([N_\R']\times\alpha))(j^*p_1^*\gamma),\]
	\[(N_\R'\wedge\alpha)(\gamma) = (\Delta\wedge ([N_\R']\times\alpha))(p_1^*\gamma),\]
	where 
	\[\xymatrix{N_\R'\times N_\R\ar@{^{(}->}[d]_j\ar[r]& N_\R'\ar@{^{(}->}[d]^i\\
		N_\R\times N_\R\ar[r]_-{p_1}& N_\R}.\]
	By \cref{another definition of wedge product}, for a generic vector $v\in N_\R\times N_\R$ and $\varepsilon>0$ small enough, we have
	\[\Delta\wedge ([N_\R']\times\alpha) = \sum\limits_{\substack{\sigma\in \mathcal{C}^l\\ (N_\R'\times\sigma)\cap(\Delta+\varepsilon v)\not=\emptyset}}\alpha_\sigma[(N_\R'\times\sigma)\cap \Delta].\]
	Notice that we can assume $v\in \{0\}\times N_\R$ since $\Delta$ is the diagonal, so we can replace $\Delta$ by $\Delta'$ in the equality above. Hence, (1) holds.
	
	(2) By \cref{product of currents}, \cref{another definition of wedge product}, we have
	\begin{align*}
		F^*\alpha&=p_{1,*}([\Delta_F]\wedge([N_\R']\times\alpha))\\
		&= p_{1,*}([\Delta_F]\wedge(\sum\limits_{\sigma\in \mathcal{C}^l}p_{2}^*\alpha_\sigma[N_\R'\times \sigma]))\\
		&= \sum\limits_{\sigma\in \mathcal{C}^l}p_{1,*}(p_{2}^*\alpha_\sigma[\Delta_F\cap (N_\R'\times \sigma)]).
	\end{align*}
Here we use the fact that $\Delta_F$ and $N_\R'\times\sigma$ intersect transversally: $\mathrm{relint}(\Delta_F)\cap \mathrm{relint}(N_\R'\times\sigma) \not= \empty$
since $F$ is surjective, $F^{-1}(\mathrm{relint}(\sigma))\not=\emptyset$, and
\[\dim\Delta_F+\dim(N_\R'\times\sigma)-\dim(\Delta_F\cap (N_\R'\times\sigma)) = \dim N_\R'+ \dim N_\R.\]
	It is not hard to check that $p_{1,*}(p_{2}^*\alpha_\sigma[\Delta_F\cap (N_\R'\times \sigma)]) = F^*\alpha_\sigma[F^{-1}(\sigma)]$ as currents: for any $\gamma \in A_c(\widetilde{\Omega}')$, we have
	\begin{align*}
		p_{1,*}(p_{2}^*\alpha_\sigma[\Delta_F\cap (N_\R'\times \sigma)])(\gamma)&= \int_{\Delta_F\cap (N_\R'\times \sigma)}p_2^*\alpha_\sigma\wedge p_1^*\gamma=\int_{F^{-1}(\sigma)}F^*\alpha_\sigma\wedge \gamma\\&= (F^*\alpha_\sigma[F^{-1}(\sigma)])(\gamma)\\
		&= [N_\sigma:dF(N_{F^{-1}(\sigma)})]\int_{\sigma}\alpha_\sigma\wedge \widetilde{F}_*\gamma\\
		&=(\alpha_\sigma[\sigma])(F_*\gamma).
	\end{align*}
The second equality from the observation that $N'\simeq \Delta_F\cap (N_\R'\times N_\R)$.
\end{proof}
\begin{remark}
	\begin{enumerate}
		\item [(1)] From (2), we know that our pull-back is exactly the one defined in \cite[Section~2.5]{mihatsch2021on} for surjective map.
	\end{enumerate}
\end{remark}

\begin{theorem}
	The presheaf $B: \mathcal{AT}op_{\Z,\R}\rightarrow \mathcal{R}ing, \ \ (N,\widetilde{\Omega})\mapsto B(\widetilde{\Omega})$ is a $C^\infty$-module. Moreover, for a morphism $F: (N',\widetilde{\Omega}')\rightarrow (N,\widetilde{\Omega})$, we have the following hold.  
	\begin{enumerate}
		\item[(1)] The $\wedge$-product on $B(\widetilde{\Omega})$ satisfies Leibniz rules with respect to $d',d'', d'_P, d''_P, \partial'$ and $\partial''$.
		\item [(2)] $F^*$ commutes with $d', d'', d_P, d''_P, \partial',\partial''$, and it is graded, i.e. $F^*(B^{p,q,l}(\widetilde{\Omega}))\subset B^{p,q,l}(\widetilde{\Omega}')$,.
		\item[(3)] (Projection formula for $F_*$) for any $\alpha\in B_c(\widetilde{\Omega}'), \beta\in B(\widetilde{\Omega})$, we have
		\[F_*(\alpha\wedge F^*\beta) = F_*\alpha\wedge\beta \in B(\widetilde{\Omega}).\]
		\item[(4)] (Projection formula for $\widehat{F}_*$) For any $\alpha\in B(N_\R'), \beta\in B(N_\R)$, we have
		\[\widehat{F}_*(\alpha\wedge F^*\beta) = \widehat{F}_*\alpha\wedge\beta \in B(N_\R).\]
	\end{enumerate}
\label{projection formula for delta forms}
\end{theorem}
\begin{proof}
	(1) It is sufficient to show the Leibniz rules with respect to $d'_P, d''_P$, and this is from the direct calculation and \cref{another definition of wedge product}.
	
	(2) This is by definition.
	
	(3) We follow the idea of the proof of \cite[Theorem~3.3~(d)]{allermann2012tropical}.  Let $p_1: N_\R'\times N_\R\rightarrow N_\R'$(resp. $q_1: N_\R\times N_\R\rightarrow N_\R$) be the projections onto the first factor, and $\Delta_F\subset N_\R\times N_\R$ (resp. $\Delta\subset N_\R\times N_\R$) the graph of $F$ (resp. the diagonal).
	Notice that, for any integral $\R$-affine morphism $G: M_\R\rightarrow N_\R\times N_\R$ and $\gamma\in B_c(M_\R)$, we have 
	\[G_*(G^*\Delta\wedge \gamma) = \Delta\wedge G_*\gamma\]
	by \cref{basic properties of pull-back1}.
	Moreover, it is easy to check that $(F\times \id)^*\Delta = \Delta_F$. We can assume that $\beta\in B_c(\widetilde{\Omega})$ by partition of unity, then $\alpha\times\beta\in B_c(\widetilde{\Omega}'\times\widetilde{\Omega})\subset B_c(N'_\R\times N_\R)$. We conclude that
	\begin{align*}
		F_{*}\alpha\wedge\beta &= q_{1,*}(\Delta\wedge (F_{*}\alpha\times\beta)) =  q_{1,*}(\Delta\wedge (F\times \id)_*(\alpha\times\beta))\\
		& = q_{1,*}(F\times \id)_*((F\times \id)^*\Delta\wedge (\alpha\times\beta))\\
		& = F_*p_{1,*}(\Delta_F\wedge(\alpha\times \beta))\\
		&= F_*(\alpha\wedge F^*\beta).
	\end{align*}

(4) This can be proved similarly as \cite[Theorem~3.3]{allermann2012tropical}, or use partition of unity and apply the projection $\pr: B(N_\R)\rightarrow B_{n+n'-\rk N'}^{p+p',q+q'}(N_\R)$ on both sides in (3) if $\alpha\in B_{n'}^{p',q'}(N_\R')$, $\beta\in B_n^{p,q}(N_\R)$. 
\end{proof}
\section{Charts and $\delta$-forms on $X^\an$}

	\label{charts on an algebraical variety}	
	In \cite{gubler2016forms} and \cite{gubler2017a}, Gubler and Gubler-K\"unnemann consider tropical charts induced by algebraic moment maps instead of general smooth charts, and they show that these two kinds of charts are equivalent, see \cite[Proposition~7.2]{gubler2016forms}. In this section, we will recall their definition and use our result to define more general $\delta$-form. Most of results in this section can be proved similarly as in \cite{gubler2017a}, we will omit the proofs or give sketches in this case.
	
	Throughout this section, $X$ denotes an algebraic variety over a complete non-archimedean field $K$, i.e. a separated scheme of finite type over $K$.
	
	\subsection{Algebraically tropical charts}
	
	Recall, an algebraic (splitting) torus $T$ over $K$ is an algebraic variety over $K$ such that $T\simeq \mathbb{G}_m^r$ for some $r\in \N^+$. Without confusion, the analytification of a torus $T$ is also denoted by $T$. The {character group of $T$} is defined by
	\[N:=\Hom_{\mathrm{An}\mathcal{G}p_{/K}}(T, \mathbb{G}_m)\simeq \Hom_{\mathcal{G}p\mathcal{S}ch_{/K}}(\mathbb{G}_m^r,\mathbb{G}_{m}),\]
	the {cocharacter group of $T$} is defined by
	\[M:=\Hom_{\mathrm{An}\mathcal{G}p_{/K}}(\mathbb{G}_m, T)\simeq \Hom_{\mathcal{G}p\mathcal{S}ch_{/K}}(\mathbb{G}_{m},\mathbb{G}_m^r).\]
	where $\mathcal{G}p\mathcal{S}ch_{/K}$ (resp. $\mathcal{G}p\mathcal{S}ch_{/K}$) is the category of analytic group spaces (resp. group schemes) over $K$. We have $N=\Hom_\Z(M,\Z)$, the dual of $M$ as $\Z$-modules. The {tropicalization map} of $T$ is
	\[\mathrm{trop}: T\rightarrow N_\R = \Hom_\Z(M,\R), \ \ x\mapsto (\varphi\mapsto -\log|\varphi^\#(z)|_x)\]
	where $\mathbb{G}_m=\Spec(K[z^\pm])$ and $\varphi^\#: K[z^\pm]\rightarrow \OO_T(T)$ is given by $\varphi$. 
	
	An affine homomorphism $\psi: T\rightarrow T'$ between tori is a group homomorphism composed with a translation on $T$. For such a group homomorphism, we have a map $M'\rightarrow M$ on character groups, which induces an affine map $N_\R\rightarrow N_\R'$.
	
	\begin{definition}
		A {\bf moment map} is an analytic morphism $\varphi: X^\an\rightarrow T\simeq\mathbb{G}^r_m$ over $K$. The {\bf tropicalization of $\varphi$} is
		\[\varphi_{\mathrm{trop}}: X\overset{\varphi}{\rightarrow} T^{\mathrm{an}} \overset{\mathrm{trop}}{\rightarrow} N_\R\simeq \R^r.\]
		
		We say that a moment map $\varphi': X\rightarrow T'$ {\bf refines} $\varphi: X\rightarrow T$ if  there is an affine homomorphism $\psi: T'\rightarrow T$ such that $\varphi=\psi\circ\varphi'$. 
		\label{definition of moment map}
	\end{definition}
	\begin{remark}
		\begin{enumerate}
			\item [(1)] After fixing coordinates $z_1,\cdots, z_r$ of $T$, we have
			\begin{align*}
				N_\R &\simeq \R^r\\
				f&\mapsto (f(\varphi_1),\cdots, f(\varphi_r)),
			\end{align*}
			\begin{align*}
				\mathrm{trop}: T^{\mathrm{an}}&\rightarrow \R^r,\\
				x&\mapsto (-\log|z_1|_x, \cdots, -\log|z_r|_x),
			\end{align*}
			where $\varphi_i^\#: K[z^{\pm 1}]\mapsto \OO_T(T)\simeq K[z_1^{\pm1},\cdots,z_r^{\pm1}], \ \ z\mapsto z_i$. 
			
			\item[(2)] Obviously $\mathrm{Trop}(\cdot): T\mapsto \mathrm{Trop}(T):=N_\R$ is a functor on the category of splitting tori. The tropicalization is functorial, i.e. for any morphism of splitting torus $f: T\rightarrow T'$, we have a commutative diagram
			\[\xymatrix{T \ar[r]^-{\trop} \ar[d]_{f}& N_\R \ar[d]^{\mathrm{Trop}(f)}\\
				T' \ar[r]_-{\trop} &N'_\R},\]
			where \begin{align*}
				\mathrm{Trop}(f): N_\R&\rightarrow N'_\R\\
				\varphi&\mapsto f^*\circ\varphi.
			\end{align*} 
			$f^*: M'\rightarrow M$ is induced by $f$.
		\end{enumerate}
	\end{remark}
	
	\begin{definition}
		A moment map $\varphi: X^\an \rightarrow T$ is called {\bf algebraic} if it is induced by a morphism $X\rightarrow T$ of algebraic varieties over $K$.
		\label{definition of algebraic moment maps} 
	\end{definition}
	\begin{remark}
		\begin{enumerate}
			\item[(1)] For finite non-empty open subsets $U_i\subset X$ with $U:=\bigcap\limits_iU_i\not=\emptyset$, if $\varphi_i: U_i^\an\rightarrow T_i$ are algebraic moment maps, then 
			\[\varphi=(\varphi_i)_i: U^\an\rightarrow \prod\limits_iT_i\]
			is algebraic and refines $U_i^\an\rightarrow T_i$ for any $i$. We have the universal property for this moment map, i.e. it is the maximal moment map which is a common refinement of all $\varphi_i$. 
			\item[(2)] If $X$ is integral, then for any open subset $U'\subset U$, we have $\varphi_{\mathrm{trop}}((U')^{\an}) = \varphi_\trop(U^\an).$ See \cite[Lemma~4.9]{gubler2016forms}.
		\end{enumerate}
	\end{remark}
	
	
	\
	
The following lemma can be easily proved.
	
	\begin{lemma}[\cite{gubler2016forms}~4.12]
		For any open affine subset $U\subset X$, we have an algebraic moment morphism
		\[\varphi_U: U\rightarrow T_U=\Spec(K[M_U]),\]
		induced by representatives $\varphi_1, \cdots, \varphi_r \in \OO_X(U)^\times$ of a basis of $M_U$,
		where $M_U: = \OO_X(U)^\times/K^\times$ is free of rank $r$. Then $\varphi_U$ is canonical up to translation by an element of $T_U(K)$, and $\varphi_U$ refines every other algebraic moment map on $U$. 
		
		In particular, the tropicalization
		\[\varphi_{U,\trop}: U^\an \overset{\varphi_U}{\rightarrow} T_U \overset{\trop}{\rightarrow} (N_U)_\R\]
		is canonical up to  integral $v(K^\times)$-affine isomorphisms, i.e. independent of the choice of $\varphi_1,\cdots, \varphi_r$, where $N_U=\Hom_\Z(M_U,\Z)$.
		\label{canonical tropical charts}
	\end{lemma}
		
		
	\begin{remark}
		\begin{enumerate}
			\item [(1)] We call such a morphism $\varphi_U$ a {\bf canonical moment map}.
			\item[(2)] The map $\trop$ is functorial: Let $f: U'\rightarrow U$ be a morphism of schemes over $K$, and $\varphi_U, \varphi_{U'}$ given as in the lemma, then there is a unique affine morphism  
		\end{enumerate}
	\end{remark}
	
	Recall that an open subset $U$ of an algebraic variety $X$ is called very affine if $U$ has a closed embedding into a multiplicative torus. 
	

	\begin{definition}
		Let $U\subset X$ be a very affine open subset, and $V\subset U^\an$ an open subset. A {\bf canonical chart} $(V,U, h)$ is a triple such that there are $\varphi_1,\cdots, \varphi_r\in \OO_X(U)^\times$ having the following properties:
		\begin{itemize}
			\item the images of $\varphi_1,\cdots, \varphi_r$ form a basis of $M_U$;
			\item $h=(-\log|\varphi_1|,\cdots, -\log|\varphi_r|)$ on $V$ (i.e. $h=\varphi_{\trop}|_V$) and there is an open subset $\Omega$ of $h(U^\an)$ with  $V=\varphi_\trop^{-1}(\Omega)$, where $\varphi=(\varphi_1,\cdots,\varphi_r): U\rightarrow T_U$.
		\end{itemize}
		In this case, we will say that $h$ is induced by the canonical moment map $\varphi$. 
		
		For two canonical charts $(V, U, h)$, $(V', U', h')$ with $V'\subset V$, $U'\subset U$, we say that $(V', U', h')$ is a {\bf refinement} of $(V, U, h)$ if there is an integral $\R$-linear map $F: \R^{r'}\rightarrow \R^{r}$ such that $h=F\circ h'$. 
		
		
		\label{definition of tropical charts} 
	\end{definition}
	\begin{remark}
		\begin{enumerate}
			\item[(1)] By \cref{canonical tropical charts}, we know that, if $(V,U,h)$ and $(V,U,h')$ are canonical charts, then there is an integral $|K^\times|$-affine bijection $F: \R^r\rightarrow \R^r$ such that $h'=F\circ h$. The converse also holds, i.e. $(V, U, F\circ h)$ is a canonical chart for any integral $|K^\times|$-affine bijection $F$. In particular, $h(U^\an)\overset{F}{\simeq} h'(U^\an)$, $h(V)\overset{F}{\simeq} h'(V)$. 
			\item[(2)] Let $(V,U, h), (V', U',h')$ be canonical charts, and $f: U'\rightarrow U$ a morphism of schemes over $K$ with $f(V')\subset V$, then there is a unique affine map $F$ such the following diagram commutes:
			\[\xymatrix{V'\ar[r]^-{h'} \ar[d]_f& \R^{r'}\ar[d]^{F}\\
				V\ar[r]_-{h}& \R^r}.\]
		\end{enumerate}
	\end{remark}

	\begin{proposition}[\cite{gubler2016forms}~Proposition~4.16]
		The following properties of canonical charts hold.
		\begin{enumerate}
			\item [(1)] The canonical charts form a topological basis of $X^\an$, i.e. for any open subset $W$ of $X^\an$ and $x\in W$, there is a canonical chart $(V,U,h)$ with $x\in V\subset W$. In particular, $X$ is covered by canonical charts. 
			\item[(2)]  For any two canonical charts $(V, U, h)$, $(V', U', h')$, there is a map $h'': V'\cap V\rightarrow \R^r$ such that $(V\cap V', U'\cap U'', h'')$ is a canonical chart. Moreover, any such canonical chart is a common refinement of both charts on $V\cap V'$.
			\item[(3)] Let  $(V, U, h)$, $(V', U', h')$ be two canonical charts with $V'\subset V$, $U'\subset U$, then $(V', U', h')$ is a refinement of $(V, U, h)$. 

		\end{enumerate}
		\label{basic properties of tropical charts}
	\end{proposition}
	
Given a tropical chart $(V, U,h)$, we have a canonical structure of tropical cycle on $h(U^\an)$, this cycle is denoted by $h_*(\mathrm{cyc}(U))$. 
	
\subsection{$\widetilde{B}(\widetilde{\Omega})$}

To apply $\delta$-forms on non-archimedean Arakelov theory, and use the idea in \cite{gubler2017a}, we don't take all $\delta$-forms on $\widetilde{\Omega}$ for technical reason, i.e. \cref{prop:separated}. Instead, for any open subset $\widetilde{\Omega}$ in an affine space $N_\R$, we consider the following subspaces
\[\widetilde{B}^{p,q,l}(\widetilde{\Omega}):=\{\sum\limits_{i}\alpha_i\wedge\beta_i\in B^{p,q}(\widetilde{\Omega})\mid \alpha_i\in A^{p_i,q_i}(\widetilde{\Omega}), \beta_i\in B^{p-p_i,q-q_i,l}(N_\R)\},\]
$\widetilde{B}(\widetilde{\Omega}) = \bigoplus\limits_{p,q}\widetilde{B}(\widetilde{\Omega})$. We notice the following properties hold:
\begin{enumerate}
	\item[(1)]  The differentials $d', d'', d_P', d_P'', \partial', \partial''$ are defined on $\widetilde{B}(\widetilde{\Omega})$ by the Leibniz rule.
	\item[(2)] For any morphism $F: (N',\widetilde{\Omega}')\rightarrow (N,\widetilde{\Omega})$ in $\mathcal{AT}op_{\Z,\R}$, we have $F^*(\widetilde{B}^{p,q,l}(\widetilde{\Omega}'))\subset B^{p,q,l}(\widetilde{\Omega})$.
\end{enumerate}


The following result can be proved similarly as \cite[Proposition~2.14]{gubler2017a}.

\begin{proposition}
	Let $F: N'_\R\rightarrow N_\R$ be an integral $\R$-affine map, and $C'=\sum\limits_{\sigma'\in \mathcal{C}'_n}m_{\sigma'}[\sigma']\in B^{0,0}(N_\R)$. Let $\alpha\in \widetilde{B}(\widetilde{\Omega})$ and $P\subset\widetilde{\Omega}$ a polyhedral set. Assume that $\Supp(F^*(\alpha)\wedge C')\cap F^{-1}(P)$ is compact, then $\Supp(\alpha\wedge \widehat{F}_*C')\cap P$ is compact, and the following hold.
	\begin{enumerate}
		\item [(1)] If $\alpha\in \widetilde{B}^{n,n}(\widetilde{\Omega})$, then
		\[\int_P \alpha\wedge \widehat{F}_*C' = \int_{F^{-1}(P)} F^*(\alpha)\wedge C'.\]
		
		\item [(2)] If $\alpha\in \widetilde{B}^{n-1,n}(\widetilde{\Omega})$ (resp. $\widetilde{B}^{n,n-1}(\widetilde{\Omega})$), then
		\[\int_{\partial P} \alpha\wedge \widehat{F}_*C' = \int_{\partial F^{-1}(P)} F^*(\alpha)\wedge C'.\]
	\end{enumerate}
	\label{projection formula for delta-preforms}
\end{proposition}

	

	\subsection{$\delta$-Forms on $X^\an$}
	

	\begin{definition}
		Let $f: X'\rightarrow X$ be a morphism of algebraic varieties over $K$. We say that charts $(V,U, h)$ and $(V',{U'}, h')$ of $X$ and $X'$ respectively are {\bf compatible with respect to $f$}, if we have $f(U')\subset U$ and $f(V')\subset V$. In this case, assume that $h=\varphi_{U,\trop}, h'=\varphi_{U',\trop}$, then there is a unique morphism $g: T_{U'}\rightarrow T_U$ such that the diagram
		\[\xymatrix{(U')^\an \ar[r]^{\varphi_{U'}} \ar[d]_{f}& (T_{U'})^\an \ar[d]^{g} \ar[r]^{\trop}& (N_{U'})_\R\ar[d]^F\\
			U^\an \ar[r]_{\varphi_U} & (T_U)^\an \ar[r]_{\trop}& (N_U)_\R}.\]
		commutes, where $F$ is the integral $\R$-affine map given by $g$.
	\end{definition}

\begin{definition}
 Let  $(V,U,h)$ a canonical chart on $X$. Set \[B^{p,q}(V,U,h): = \widetilde{B}^{p,q}(\widetilde{\Omega})/N^{p,q}(V,U,h),\]
where  \[N^{p,q}(V,U,h):=\left\{\widetilde{\alpha}\in \widetilde{B}^{p,q}(\widetilde{\Omega})\,\middle\vert\,\parbox[c]{.6\linewidth}{ $F^*\widetilde{\alpha}\wedge h_*(\mathrm{cyc}(U')) = 0 \in D^{r-n+p,r-n+q}(\widetilde{\Omega}')$ for any compatible tropical chart $f: (V', {U'}, h')\rightarrow (V, U, h)$}\right\},\]
$F$ is given by $f$ such that $h\circ f=F\circ h'$, and $\widetilde{\Omega}\subset N_{U,\R}$ (resp. $\widetilde{\Omega}'\subset N'_{U,\R}$) is an open subset such that $h(V)=\widetilde{\Omega}\cap h(U^\an)$ (resp. $h'(V')=\widetilde{\Omega}'\cap h((U')^\an)$) with $F(\widetilde{\Omega}')\subset\widetilde{\Omega}$. Similarly, we have the notion $B^{p,q,l}(V,U,h)$ and $B(V,U,h)=\prod\limits_{p,q}B^{p,q}(V,U,h)$.

By \cref{projection formula for delta forms} and \cref{boundary derivatives are null on delta preforms}, we know that $d', d'', d_P', d_P'', \partial', \partial''$ are well defined on $B(V,U,h)$.
\end{definition}
\begin{remark}
	\begin{enumerate}
		\item [(1)] By partition of unity, the definition is independent of the choice of $\widetilde{\Omega}$ and $\widetilde{\Omega}'$.
	\end{enumerate}
\end{remark}

 Unlike \cite[4.6]{gubler2017a}, we have differentials on $B^{p,q}(V,U,h)$, so we don't need to define  $Z(V,U,h)$ and $AZ(V,U,h)$.
 
 Given canonical charts $(V, U, h)$, $(V', U', h')$, and $s\in B(V,U,h), s'\in B(V',U',h')$, we write $s|_{V\cap V'}=s'|_{V\cap V'}$ if there is a canonical chart $(V\cap V', U\cap U', h'')$ compatible with both such that the images of $s, s'$ coincide in $B(V\cap V', U\cap U', h'')$. Obviously, this is independent of the choice of $h''$. 

\begin{definition}
Let $W\subset X^\an$ be an open subset. We set
\[B^{p,q}(W)=\left\{(s_i)_{i}\in \prod\limits_iB^{p,q}(V_i, U_i, h_i)\,\middle\vert\,\parbox[c]{.5\linewidth}{ $W=\bigcup\limits_iV_i$ is an open covering of $W$, $(V_i, U_i, h_i)$ are canonical charts, and $s_i|_{V_{i}\cap V_j}= s_j|_{V_{i}\cap V_j}$}\right\}/\sim,\]
where $(s_i)_i \sim (s_j')_{j}$ if for any $i,j$, we have $s_i|_{V_{ijk}} = s'_j|_{V_{ijk}}$ for a covering $V_i\cap V_j'=\bigcup\limits_{k}V_{ijk}$ and charts $(V_{ijk}, U_{ijk}, h_{ijk})$.
\end{definition}
\begin{remark}
	\begin{enumerate}
		\item [(1)] Let $f: X'\rightarrow X$ a morphism of algebraic varieties over $K$, and $V\subset X^\an$, $V'\subset (X')^\an$ open subsets with $f(V')\subset V$. Let $s=(s_i)_i\in \mathcal{F}(V)$ with $s_i\in B(V_i)$ for some canonical charts $(V_i, U_i, h_i)$. Then $f^*s\in B(V')$ is given as follows:  choose canonical charts $\{(V_{ij}', U_{ij}', h_{ij}')\}_{i\in I, j\in J_i}$ such that $f^{-1}(U_i)=\bigcup\limits_{j\in J_i}U_{ij}'$, $f^{-1}(V_i)= \bigcup\limits_{j\in J_i}V_{ij}'$, then there is a unique integral $\R$-affine map such that
		\[\xymatrix{(U'_{ij})^\an \ar[r]^{h_{ij}'} \ar[d]_{f}& (N_{U_{ij}}')_\R\ar[d]^{F_{ij}}\\
			U_i^\an \ar[r]_{h_i} & (N_{U_i})_\R},\]
		and $f^*s$ is given by $(F_{ij}^*s_{i})_{i,j}\in\prod\limits_{i,j}B(V_{ij}', U_{ij}', h_{ij})$.
		\item[(2)] It can be proved that $B^{p,q}(W)=0$ if $\dim_KW< \max\{p,q\}$, see \cite[Corollary~5.6]{gubler2017a}
	\end{enumerate}
\end{remark}



The following proposition is the main reason we take $\widetilde{B}(\widetilde{\Omega})$ to define $\delta$-forms rather than $B(\widetilde{\Omega})$, see \cite[Proposition~4.21]{gubler2017a} for the proof.

\begin{proposition}
	Assume that $X$ is integral. Given a canonical chart $(V,U,h)$ on $X$, the natural algebra homomorphism
	\[B^{p,q}(V,U,h)\rightarrow B^{p,q}(V)\]
	is injective. 
	\label{prop:separated}
\end{proposition}

\subsection{Integration of $\delta$-forms}

	In the rest of this section, we assume that $X$ is integral algebraic variety of pure dimension $n$.
	
	\begin{definition}
		Let $W$ be an open subset of $X^\an$. Let $\alpha\in B_{c}^{n,n}(W)\subset B_{c}^{n,n}(X^\an)$. A non-empty very affine open subset $U$ of $X$ is called a {\bf very affine chart of integration for $\alpha$} if there is some $\alpha_U\in B^{n,n}(U^\an, U, h)$ such that $\alpha|_{U^\an}$ is given by $\alpha_U$, where $(U^\an, U, h)$ is a canonical chart. Such a chart exists and we have $\Supp(\alpha)\subset U^\an$ by a similar proposition as \cite[Proposition~5.7]{gubler2017a}.
		
		We know that chart of integration always exists, and $\alpha_U$ is unique. We define the {\bf integral of $\alpha$ over $W$} by
		\[\int_{W}\alpha: = \int_{N_{U,\R}}\widetilde{\alpha}_U\wedge h_*(\mathrm{cyc}(U)),\]
		where $\alpha_U\in B^{n,n}(N_{U,\R})$ and  the right-hand side is given in \cref{definition of integration for polyhedral supercurrents}. $\widetilde{\alpha}_U\wedge h_*(\mathrm{cyc}(U))$ has compact support by \cite[Proposition~4.21]{gubler2017a}.
	\end{definition}

From the definition, we can show the Stokes' formula and Green's formula for $d'_P, d''_P$ by  \cref{stokes' formula for polyhedral supercurrents} and \cref{Green' formula for delta-forms}. \cref{minor stokes' formula for d'd''} implies the following theorem.

\begin{theorem}
	For any $\omega\in B_c^{2n-1}(X)$, we have
	\[\int_Xd'\omega=\int_Xd''\omega=0.\]
	It is similar for $d'_P, d''_P, \partial'$ and $\partial''$.
	\label{stokes formula for delta form on analytic space}
\end{theorem}

\subsection{$\delta$-currents}
We have the Schwartz topology on $B^{p,q,l}_c(W)$, see \cite[6.1]{gubler2017a}.
\begin{definition}
	Let $W\subset X^\an$ be an open subset. A {\bf $\delta$-current} on $W$ is a real continuous linear functional $T: {B}_{c}(W)\rightarrow \R$. We denote the space of $\delta$-currents on $W$ by $E(X)$.
	\label{currents}
\end{definition}
\begin{remark}
	\begin{enumerate}
		\item [(1)] We have a natural bigrading $E(W) = \bigoplus\limits_{p,q}(E)^{p,q}(X)$, where $E^{p,q}(W)$ is the set of continuous linear functional on $B_c^{n-p, n-q}(W)$. 
\item[(2)] Let $D(W)$ the set of currents on $W$ defined in \cite[4.2]{chambert2012formes} or \cite[6.2]{gubler2016forms}. We have a linear map $E(W)\rightarrow D(W)$. We can define differentials $d', d'', d_P', d_P'', \partial',\partial'': E(W) \rightarrow E(W)$ similarly as in \cite[6.9]{gubler2017a}. For $T\in E^{p,q}(W)$ and $\alpha\in B^{p',q'}(W)$, we have wedge product $T\wedge\alpha\in E^{p+p',q+q'}(W)$ defined as
\[(T\wedge \alpha)(\gamma): = T(\alpha\wedge\gamma)\]
for any $\gamma\in B^{n-p-p',n-q-q'}_c(W)$. 

\item[(3)] We have an $\R$-linear map $[\cdot]: B^{p,q}(W)\rightarrow E^{p,q}(W)$ defined as follows: for any $\omega\in B^{p,q}(W)$, we have 
\begin{align*}
	[\omega]: {B}_{c}^{n-p,n-q}(W)& \rightarrow \R,\\
	\gamma&\mapsto \int_W\omega\wedge\gamma.
\end{align*}
By \cref{stokes formula for delta form on analytic space}, we know that $[\cdot]$ commutes with $d'$, $d''$, $d_P', d_P'', \partial'$ and $\partial''$. 
	\end{enumerate}
\end{remark}

Every $\delta$-form 

\begin{proposition}
	Let $W\subset X^\an$ be an open subset. For each $\alpha\in B^{n,n}(W)$ there is a unique signed Radon measure $\mu_\alpha$ on $W$ such that
	\[\int_{W}f\cdot\alpha=\int_Wf\mu_\alpha\]
	for all $f\in C^\infty_{c}(W)$.
	\label{cor:radonmeasureassociatedtodeltaform}
\end{proposition}

\begin{proposition}
	Let $W\subset X^\an$ be an open subset. Let $f\in C(W)$. Then the map
	\[[f]: B_{c}^{n,n}(W)\rightarrow \R, \ \ \eta\mapsto\int_Wf\mu_\eta\]
	is a $\delta$-current.
\end{proposition}

We can give the form of $\mu_\alpha$ for $\alpha\in B^{0,0,n}(X)$. The following result extends \cite[Proposition~6.9.2]{chambert2012formes} for case of the proper algebraic varieties.

\begin{proposition}
	Let $W\subset X^\an$ be an open subset, and $\alpha\in B_c^{0,0,n}(W)$. Then there are $c_x\in \R$ for each $x\in \Supp(\alpha)$ such that $c_x=0$ for all but finitely many $x$ and for any $f\in {B}^{0,0}_c(X)$, we have
	\[\int_Xf\alpha = \sum\limits_{x\in \Supp(\alpha)}c_xf(x),\]
	i.e. $\mu_\alpha = \sum\limits_{x\in \Supp(\alpha)}c_x\delta_x$, where $\mu_\alpha$ is the signed Radon measure associated to $\omega$ given in \cref{cor:radonmeasureassociatedtodeltaform}. 
	\label{prop:radonmeasuregivenbytropicalcycleofmaximaldimension}
\end{proposition}
\begin{proof}
	We can take a very affine chart $U$ of integration for $\alpha$. Assume that $\alpha_U\in B^{n,n}(U^\an, U, h)$ such that $\alpha|_{U^\an}$ is given by $\alpha_U$, where $(U^\an, U, h)$ is a canonical chart. Then $\alpha_U\wedge h_*(\mathrm{cyc}(U))=\sum\limits_{s\in S}c_s[s]\in N_\R$ for a finite set $S\subset N_{U,\R}$. Let $S_h(U)$ be the skeleton of $h$. Notice that $\Supp(\alpha)\subset S_h(U)$ and $h^{-1}(S)\subset\Supp(\alpha)$ by \cite[Proposition~4.21]{gubler2017a}. By \cite[Corollary~3.16]{gubler2021forms}, we know that $h^{-1}(S)$ is finite. After considering bump functions around each point in $h^{-1}(S)$, we have the result.
\end{proof}

	\section{The Poincar\'e-Lelong formula}
\label{The Poincare-Lelong formula and the first Chern delta-forms}

In this section, we will reprove the Poincar\'e-Lelong formula, then our definition of $\delta$-forms allows us to define the first Chern $\delta$-forms of a metrized line bundle with a piecewise smooth metric.
We fix a complete non-archimedean field $K$.
	
 For readers' convenience, we recall \cite[Proposition~4.6.6]{chambert2012formes}.  
	
	\begin{proposition}[\cite{chambert2012formes}~Proposition~4.6.6]
		Assume that $X$ is a $K$-affinoid space of pure dimension $n$. Let $g: X\rightarrow \R^r$ be a smooth tropicalization and $f\in \OO_X(X)$. Then there is a real number $r_0>0$ and  $C=\sum\limits_{\sigma\in \mathcal{C}_{n-1}}m_\sigma[\sigma]\in F_{n-1}(\R^r)$ such that all polyhedra in $\mathcal{C}$ are polytopes with the following properties:
		\begin{enumerate}
			\item [(a)] For every $t$ in the closed ball in $(\mathbb{A}^1)^\an$ with center $0$ and radius $r_0$, the $(n-1)$-skeleton of $g_*(\mathrm{cyc}(X_t))$ endowed with the canonical tropical weights is equal to $C=\sum\limits_{\sigma\in \mathcal{C}_{n-1}}m_\sigma[\sigma]$;
			\item[(b)] For every closed interval $I\subset(0,r_0]$, the $n$-skeleton of $h_*(\mathrm{cyc}(X_I))$ endowed with the canonical tropical weight is equal to $C\times [-\log(I)]$ as a product of polyhedral complexes with weights,
		\end{enumerate}
		where $X_t$ is the fiber of $X$ over $t\in (\mathbb{A}^1)^\an$ with respect to $f: X\rightarrow (\mathbb{A}^1)^\an$, and $h:=(g,-\log|f|): X_I\rightarrow \R^{r+1}$ with $X_I:=|f|^{-1}(I)$.
		\label{canonical cycle for subspace}
	\end{proposition}

\
	
	
In general, given two polyhedral complex with smooth forms $\alpha=\sum\limits_{\sigma\in\mathcal{C}_1^{l_1}} \alpha_\sigma[\sigma], \beta=\sum\limits_{\sigma\in\mathcal{C}_2^{l_2}} \beta_\sigma[\sigma]\in P(\widetilde{\Omega})$, the wedge product $\alpha\wedge \beta$ is not defined.  But if $\mathcal{C}_1^{l_1}, \mathcal{C}_2^{l_2}$ intersect transversally, then we can set 
	\[\alpha\wedge\beta = \sum\limits_{\substack{(\sigma_1,\sigma_2)\in \mathcal{C}_1^{l_1}\times \mathcal{C}_2^{l_2}\\ \sigma_1\cap\sigma_2\not=\emptyset}}(\alpha_{\sigma_1}\wedge\beta_{\sigma_2})[\sigma_1\cap \sigma_2].\]
This definition is independent of the choice of $\mathcal{C}_1^{l_1}, \mathcal{C}_2^{l_2}$ that intersect transversally.

With this notation, \cref{canonical cycle for subspace} implies the following corollary.
	\begin{corollary}
		Keep the assumption in \cref{canonical cycle for subspace}. Let $V(f)\subset X$ be the Zariski-closed subspace given by $f$. Then for any $s>0$ small enough and any closed interval $I\subset(0,\infty)$ with $s\in \mathrm{Int}(I)$, we have $H_{-\log(s)}$, $h_*(\mathrm{cyc}(X_{I}))$ intersect transversally, and
			\[g_*(\mathrm{cyc}(V(f))) = p_*(H_{-\log(s)}\wedge h_*(\mathrm{cyc}(X_{I}))),\]
		on $\R^r$, where $H_{-\log(s)}$ is the hyperplane in $\R^{r+1}$ given by $x_{r+1}=-\log(s)$ which and $p:\R^{r+1}\rightarrow \R^r, \ \ (x_1,\cdots, x_{r+1})\mapsto (x_1,\cdots,x_r)$. 
		\label{canonical cycle for subspace2}
	\end{corollary}

	Obviously, the corollary above holds if $X$ is an affine variety.
	
	\begin{theorem}[The Poincar\'e-Lelong formula]
		Let $X$ be an integral algebraic variety over $K$. For any rational function $f\in K^*(X)$, the Poincar\'e-Lelong equation\[\delta_{[\mathrm{div}(f)]} = d'd''[\log|f|]\]
		holds in $E^{1,1}(X^\an)$.
		\label{poincare-lelong equation}
	\end{theorem}
	\begin{proof}
		For any $\gamma\in {B}^{n-1,n-1}_{c}(X)$, by \cite[Proposition~5.7]{gubler2017a}, we can take a canonical chart $(U^\an, U, g)$ such that $\gamma$ is given by $\gamma_U\in B^{n-1,n-1}(N_{U,\R})$ with $\Supp(d''\gamma)\subset U^\an$, and for any irreducible component $Z$ of $V(f)$ we have $\Supp(\gamma|_Z)\subset (Z\cap U)^\an$. Moreover, since $\gamma$ has compact support, we can take an open subset $\widetilde{\Omega}\subset N_{U,\R}$ such that $\Supp(d''\gamma), \Supp(\gamma|_Z)\subset g^{-1}(\widetilde{\Omega})$ and $W:=g^{-1}(\overline{\widetilde{\Omega}})$ is affinoid. We know that $g_*(\mathrm{cyc}(V(f)\cap W))\in B(\widetilde{\Omega})$, i.e. $C\in B(\widetilde{\Omega})$ where $C$ is given in \cref{canonical cycle for subspace} when we replace $X$ by $W$.
		
		Since $\dim V(f) <n$, we have $\Supp(d''\gamma)\cap V(f)=0$. Then there is $s>0$ small enough such that $|f(x)|>s$ for any $x\in \Supp(d''\gamma)$. Set $H_{-\log(s)}$ is the hyperplane in $\R^{r+1}$ given by $x_{r+1}=-\log(s)$, $\phi= \max\{x_{r+1},-\log(s)\}, \psi=\min\{x_{r+1}, -\log(s)\}: \R^{r+1}\rightarrow \R$ and $I=(s',\infty)$ for some $s'<s$. We use the notation in \cref{canonical cycle for subspace}. We take $\widetilde{\Omega}'=\widetilde{\Omega}\times(s'', \infty)$ where $s'<s''<s$, then $h_*(\mathrm{cyc}(W_{I}))=C\times I\in B(\widetilde{\Omega})$. By \cref{tropical poincare-lelong formula}, we have 
		\begin{align*}
			\int_{V(f)}\gamma &= \int_{\widetilde{\Omega}}\gamma_U\wedge g_*(\mathrm{cyc}(V(f)\cap W))\\
			&=\int_{\widetilde{\Omega}}\gamma_U\wedge p_*(\mathrm{div}(\phi)\cdot h_*(\mathrm{cyc}(W_{I})))\\
			& = \int_{\widetilde{\Omega}'}p^*\gamma_U\wedge (\mathrm{div}(\phi)\cdot h_*(\mathrm{cyc}(W_{I})))\\
			&= \int_{\widetilde{\Omega}'}(\mathrm{div}(\phi)\cdot p^*\gamma_U)\wedge  h_*(\mathrm{cyc}(W_{I}))\\
			&= \int_{\widetilde{\Omega}'}(-\mathrm{div}(\psi)\cdot p^*\gamma_U)\wedge  h_*(\mathrm{cyc}(W_{I}))\\
			&= \int_{\widetilde{\Omega}'}(-d'd''\psi\wedge p^*\gamma_U)\wedge h_*(\mathrm{cyc}(W_{I}))\\
			&= \int_{\widetilde{\Omega}'}(-\psi\wedge d'd''p^*\gamma_U)\wedge h_*(\mathrm{cyc}(W_{I}))\\
			&= \int_{\widetilde{\Omega}'}-p_*\psi\wedge d'd''\gamma_U\wedge p_*h_*(\mathrm{cyc}(W_{I}))\\
			&= \int_{X\setminus V(f)}-\min\{-\log|f|, -\log(s)\}\wedge d'd''\gamma\\
			&= \int_{X\setminus V(f)}\log|f|\wedge d'd''\gamma,
		\end{align*}
The third equality hold since $p: H_{-\log(s)} \rightarrow \R^r$ is injective, $\Supp(p^*\gamma_U)\cap H_{-\log(s)} \subset \widetilde{\Omega}'$ is compact.
	\end{proof}

With this theorem, the first Chern $\delta$-current is defined for any continuous metric on a line bundle, see \cite[7.7]{gubler2017a}. In particular, it is a $\delta$-form if the metric is piecewise smooth, see \cite[Definition~8.5]{gubler2017a} for the definition of piecewise smooth metrics.



\addcontentsline{toc}{section}{Acknowledgements}	
	\section*{Acknowledgements}
	The author would like to thank my host professor, Yigeng Zhao for his encouragement, support and valuable suggestions. He would also like to thank Antoine Ducros, Walter Gubler and Michael Temkin for their patience and answering questions when he worked on this problem. This research is supported by postdoctoral research grant.

	\address
\end{document}